\documentclass[10pt]{amsart}
\usepackage{amsxtra, amsfonts, amsmath, amsthm, amstext, amssymb, amscd, mathrsfs, verbatim, color}
\usepackage{threeparttable}
\usepackage{mathtools}
\usepackage{tikz} 
\usepackage[ansinew]{inputenc}\usepackage[T1]{fontenc}
\usepackage[all,cmtip]{xy}
\usepackage[ngerman,french,english]{babel}

\theoremstyle{plain}

\title{minimalSequences2}

\date{October 2025}

\usepackage{amsxtra, amsfonts, amsmath, amsthm, amstext, amssymb, amscd, mathrsfs, verbatim, color}
\usepackage{threeparttable}
\usepackage{mathtools}
\usepackage[ansinew]{inputenc}\usepackage[T1]{fontenc}
\usepackage[all,cmtip]{xy}
\usepackage{tikz}
\usetikzlibrary {positioning}

\theoremstyle{plain}

\newtheorem{theorem}{Theorem}[section]
\newtheorem{lemma}[theorem]{Lemma}
\newtheorem{definition-theorem}[theorem]{Definition-Theorem}
\newtheorem{proposition}[theorem]{Proposition}
\newtheorem{corollary}[theorem]{Corollary}

\newtheorem{conjecture}[theorem]{Conjecture}

\theoremstyle{definition}
\newtheorem{definition}[theorem]{Definition}
\newtheorem{example}[theorem]{Example}
\newtheorem{remark}[theorem]{Remark}
\newtheorem{notation}[theorem]{Notation}
\newcommand \bth[1] { \begin{theorem}\label{t#1} }
\newcommand \ble[1] { \begin{lemma}\label{l#1} }

\newcommand \bpr[1] { \begin{proposition}\label{p#1} }
\newcommand \bco[1] { \begin{corollary}\label{c#1} }
\newcommand \bde[1] { \begin{definition}\label{d#1}\rm }
\newcommand \bex[1] { \begin{example}\label{e#1}\rm }
\newcommand \bre[1] { \begin{remark}\label{r#1}\rm }

\newcommand \bnota[1] {\begin{notation}\label{n#1}\rm }
\newcommand {\ele} { \end{lemma} }

\newcommand {\epr} { \end{proposition} }
\newcommand {\eco} { \end{corollary} }
\newcommand {\ede} { \end{definition} }
\newcommand {\eex} { \end{example} }
\newcommand {\ere} { \end{remark} }
\newcommand {\enota} { \end{notation} }

\begin{document}
\title[Properties of minimal sequences]{Construction of simple quotients of Bernstein-Zelevinsky derivatives and  highest derivative multisegments III: properties of minimal sequences} 
\date{\today}

\author[Kei Yuen Chan]{Kei Yuen Chan}
\address{
Department of Mathematics, The University of Hong Kong
}
\email{kychan1@hku.hk}

\begin{abstract}
Let $F$ be a non-Archimedean local field. For an irreducible smooth representation $\pi$ of $\mathrm{GL}_n(F)$ and a multisegment $\mathfrak m$, one associates a simple quotient $D_{\mathfrak m}(\pi)$ of a Bernstein-Zelevinsky derivative of $\pi$. In the preceding article, we showed that
\[  \mathcal S(\pi, \tau) :=\left\{ \mathfrak m : D_{\mathfrak m}(\pi)\cong \tau \right\} ,
\]
has a unique minimal element under the Zelevinsky ordering, where $\mathfrak m$ runs for all multisegments. The main result of this article includes commutativity and subsequent property of the minimal sequence. At the end of this article, we conjecture some module structure arising from the minimality.
\end{abstract}
\maketitle


\part{Introduction and preliminaries}

\section{Introduction}


\subsection{The poset $\mathcal S(\pi, \tau)$ and the minimal sequence} \label{ss main results}


Let $F$ be a non-Archimedean local field. Let $G_n=\mathrm{GL}_n(F)$, the general linear group over $F$.  Fix a cuspidal representation $\rho$ throughout the whole article. All the representations we consider are smooth and over $\mathbb C$. Let $\mathrm{Irr}(G_n)$ be the set of irreducible representations of $G_n$. We shall usually not distinguish isomorphic irreducible representations. For a representation $\pi_1$ of $G_{n_1}$ and a representation $\pi_2$ of $G_{n_2}$, denote the normalized parabolic induction by $\pi_1 \times \pi_2$. 

The complex representation theory of $\mathrm{GL}_n(F)$ has fruitful study in the literature, back to \cite{Ze80}. One central object is the so-called multisegments that parametrize irreducible representations of $\mathrm{GL}_n(F)$. Recently, \cite{Ch22+d, Ch22+e} found interesting combinatorics on multisegments arising from certain sequence of derivatives  of essentially square-integrable representations in the content of Bernstein-Zelevinsky derivatives \cite{Ch22+d, Ch22+e}. The notion of highest derivative multisegments is a key combinatorial notion to encode information for derivatives and is quite computable (c.f. explicit algorithms in \cite{LM16, CP25}). These results are motivated and have applications to quotient branching laws \cite{Ch22+b}, and also attempt to generalize some aspects of some  segment cases in \cite{AL23, LM22} to multisegment cases.

This article continues to study those sequences of derivatives. A main goal is to establish and prove some general properties of a sequence of derivatives, and explain why those properties are natural from some structure of Jacquet modules. In order to define the notion of derivatives precisely, we need more notations, following Zelevinsky \cite{Ze80}.

Let $\nu: G_n \rightarrow \mathbb C^{\times}$ be the character $\nu(g)=|\mathrm{det}(g)|_F$, where $|.|_F$ is the normalized absolute value for $F$. Let $\mathrm{Irr}_{\rho}(G_n)$ be the set of irreducible representations of $G_n$ which are an irreducible constitutent of $\nu^{a_1}\rho \times \ldots \times \nu^{a_r}\rho$ for some integers $a_1, \ldots, a_r$. Let $\mathrm{Irr}_{\rho}=\sqcup_n \mathrm{Irr}_{\rho}(G_n)$.

We now define some combinatorial objects to parametrize and study representations. For $a,b \in \mathbb Z$ with $b-a \in \mathbb Z_{\geq 0}$, we call $[a,b]_{\rho}$ to be a {\it segment} (associated to $\rho$). We also set $[a,a-1]_{\rho}=\emptyset$ for $a \in \mathbb Z$. For a segment $\Delta=[a,b]_{\rho}$, we write $a(\Delta)=a$ and $b(\Delta)=b$. We also write $[a]_{\rho} :=[a,a]_{\rho}$. Let $\mathrm{Seg}_{\rho}$ be the set of segments. A {\it multisegment} (associated to $\rho$) is a multiset of non-empty segments. Let $\mathrm{Mult}_{\rho}$ be the set of multisegments. For each segment $\Delta \in \mathrm{Seg}_{\rho}$, we associate to the unique essentially square-integrable representation, $\mathrm{St}(\Delta)$, of $G_n$ as one of the irreducible composition factors in $\nu^a\rho\times \ldots \times \nu^b\rho$ (also see \cite[Section 2.6]{Ch22+d}). One may refer to \cite{Ze80} a more general notion of multisegments, which we shall not use in this article.

For $\pi \in \mathrm{Irr}_{\rho}(G_n)$ and a segment $\Delta \in \mathrm{Seg}_{\rho}$, there is at most one irreducible module $\tau \in \mathrm{Irr}_{\rho}(G_{n-i})$ such that 
 \[   \pi \hookrightarrow \tau \times  \mathrm{St}(\Delta).
\]
 If such $\tau$ exists, we denote such $\tau$ by $D_{\Delta}(\pi)$ and call $\Delta$ to be {\it admissible} to $\pi$. Otherwise, we set $D_{\Delta}(\pi)=0$. We shall refer $D_{\Delta}$ to be a {\it derivative}. 

A sequence of segments $[a_1,b_1]_{\rho}, \ldots, [a_k,b_k]_{\rho}$ (all $a_j, b_j\in \mathbb Z$) is said to be in an {\it ascending order} if for any $i\leq j$, either $[a_i,b_i]_{\rho}$ and $[a_j,b_j]_{\rho}$ are unlinked; or $a_i<a_j$. For a multisegment $\mathfrak n \in \mathrm{Mult}_{\rho}$, which we write the segments in $\mathfrak n$ in an ascending order $\Delta_1, \ldots, \Delta_k$. Define 
\[  D_{\mathfrak n}(\pi):=D_{\Delta_k}\circ \ldots \circ D_{\Delta_1}(\pi) .
\]
The derivative is independent of the choice of an ascending order \cite{Ch22+d}. In particular, one may choose an ordering such that $a_1\leq \ldots \leq a_k$. We say that $\mathfrak n$ is {\it admissible} to $\pi$ if $D_{\mathfrak n}(\pi)\neq 0$. We refer the reader to \cite{LM16, Ch22+, Ch22+d} for more theory on derivatives.





For $\pi \in \mathrm{Irr}_{\rho}$, denote its $i$-th Bernstein-Zelevinsky derivatives by $\pi^{(i)}$. We shall refer the reader \cite{Ze80, Ch21, Ch22+d} for the precise definition, and the main discussions and proofs will not involve the use of Bernstein-Zelevinsky derivatives. The main relation of derivatives and Bernstein-Zelevinsky derivatives is that $D_{\mathfrak n}(\pi)$ is a simple quotient of $\pi^{(i)}$ \cite{Ch22+d}, where $i=l_{abs}(\mathfrak n)$  (see Section \ref{ss notations} for the notation $l_{abs}$). The goal of this series of articles \cite{Ch22+d, Ch22+e} is to study constructions of Bernstein-Zelevinsky derivatives from $D_{\mathfrak n}(\pi)$. In particular, \cite{Ch22+e} studies the following poset: for a simple quotient $\tau$  of $\pi^{(i)}$,
\[ \mathcal S(\pi, \tau) := \left\{ \mathfrak n \in \mathrm{Mult}_{\rho}: D_{\mathfrak n}(\pi) \cong \tau \right\} .
\]
 The ordering  $\leq_Z$ on $\mathcal S(\pi, \tau)$ is the Zelevinsky ordering (see Section \ref{ss zel ordering}). We recall two fundamental combinatorial structure on the set $\mathcal S(\pi, \tau)$:

\begin{theorem} \cite[Theorem 1.1]{Ch22+e} \label{thm convex derivatives}
We use the notatons above. Let $\mathfrak n_1, \mathfrak n_2 \in \mathcal S(\pi, \tau)$ with $\mathfrak n_1\leq_Z \mathfrak n_2$. If $\mathfrak n \in \mathrm{Mult}_{\rho}$ with $\mathfrak n_1 \leq_Z \mathfrak n \leq_Z \mathfrak n_2$, then $\mathfrak n \in \mathcal S(\pi, \tau)$.
\end{theorem}

\begin{theorem} \cite[Theorem 1.2]{Ch22+e} \label{thm unique minimal}
We use the notatons above. Suppose $\mathcal S(\pi, \tau)\neq \emptyset$. Then $\mathcal S(\pi, \tau)$ has a unique $\leq_Z$-minimal element.
\end{theorem}






For $\pi \in \mathrm{Irr}_{\rho}$, a multisegment $\mathfrak n \in \mathrm{Mult}_{\rho}$ is said to be {\it minimal} to $\pi$ if $D_{\mathfrak n}(\pi)\neq 0$ and $\mathfrak n$ is $\leq_Z$-minimal in $\mathcal S(\pi, D_{\mathfrak n}(\pi))$. We shall sometimes refer such $\mathfrak n$ to be the {\it minimal multisegment} or {\it minimal sequence} (of derivatives) of $\mathcal S(\pi, \tau)$.

\subsection{Main results}
The main goal of this article  is to obtain some properties of the minimal sequence, which are useful in \cite{Ch22+b}. We also provide some (partly conjectural) representation theoretic interpretations in Part \ref{part rep theoy asp}. The main results are the following subsequence and commutativity phenomenons:


\begin{theorem} (=Theorem \ref{thm subsequent minimal}) \label{thm subsequent minimal intro}
Let $\pi \in \mathrm{Irr}_{\rho}$. If $\mathfrak n \in \mathrm{Mult}_{\rho}$ is {\it minimal} to $\pi$, then any submultisegment $\mathfrak n'$ of $\mathfrak n$ is also minimal to $\pi$ and in particular, $D_{\mathfrak n'}(\pi)\neq 0$.
\end{theorem}

\begin{theorem} (=Theorem \ref{thm minimal and commut 2}) \label{thm commute and minimal}
Let $\pi \in \mathrm{Irr}_{\rho}$. If $\mathfrak n \in \mathrm{Mult}_{\rho}$ is minimal to $\pi$, then for any submultisegment $\mathfrak n'$ of $\mathfrak n$, we have:
\begin{enumerate}
\item[(1)] $\mathfrak n-\mathfrak n'$ is minimal to $D_{\mathfrak n'}(\pi)$; and
\item[(2)] $D_{\mathfrak n-\mathfrak n'}\circ D_{\mathfrak n'}(\pi) \cong D_{\mathfrak n}(\pi)$. 
\end{enumerate}
\end{theorem}

By using Theorem \ref{thm commute and minimal} multiple times, we have:

\begin{corollary}
Let $\pi \in \mathrm{Irr}_{\rho}$. Let $\mathfrak n \in \mathrm{Mult}_{\rho}$ be minimal to $\pi$. Write the segments in $\mathfrak n$ as $\left\{ \Delta_1, \ldots, \Delta_r \right\}$ in any order. Then, 
\[  D_{\Delta_r}\circ \ldots \circ D_{\Delta_1}(\pi) \cong D_{\mathfrak n}(\pi) .
\] 
\end{corollary}

One important ingredient in studying the commutativity is a notion of $\eta$-invariants (see Definition \ref{def eta invariant}), which also plays an important role in studying "left-right" commutativity in \cite{Ch22+c}. It seems that such commutativity phenomenon in derivatives plays a crucial role in quotient branching laws \cite{Ch22+b}, and the above results are not quite expected before, not even that the minimal sequences seem to be known before. It is interesting to see whether there are some deeper reasons behind such commutativity. The embedding model in Conjecture \ref{conj minimal model} is an attempt to provide some explanations on that, and the interplay with the removal process is conjectured in Section \ref{ss conjectures}.

In order to demonstrate the above non-straightforward results, it is unavoidable to work on some details. The main idea of the proofs for Theorems \ref{thm subsequent minimal intro} and \ref{thm commute and minimal} is to reduce checking elementary intersection-union processes by Theorem \ref{thm convex derivatives}. Then one uses some more basic commutativity for two segment case (e.g. Proposition \ref{prop dagger property 2}) to reduce to three segment cases in Section \ref{s three segment cases}. There are also interactions between commutativity and minimality in the proofs.

While our proofs are largely combinatorial in nature, a motivation comes from a simple example from Lemma \ref{lem commut for max case} below. There are some further results on minimality such as a   construction from the removal process, which will be explored in another sequel \cite{Ch26+}.



The commutativity and minimality also play important roles in the branching law \cite{Ch22+b}. The uniqueness of minimality is closely related to the layer of Bernstein-Zelevinsky filtration determining a branching law \cite{Ch22+b}.

We remark that it has been known that the representation theory of reductive $p$-adic groups has been known to have connections to many objects objects of type A. It is a classical result of Borel-Casselman to transfer results to the module category of affine Hecke algebras (see e.g. \cite{Ro86, CS19}), and then transfer to the module category of quantum affine algebras of type A (e.g \cite{Ch86, CP96}). One may expect that the results can be translated into those setting in suitable manners, and may have some other interpretations. It is also an interesting question to see whether the results have a geometric counterpart in the geometry for ABV packets, see e.g. a recent study in \cite{CR24} as well as in \cite{AL23}. 

We also remark that Bernstein-Zelevinsky derivatives can be viewed as certain kinds of degenerate Whittaker models. While the sequences of derivatives study structure of Bernstein-Zelevinsky derivatives, the information is more sensitive in the representation category rather than simply in its Grothendieck group. This is in contrast to some study on some behaviours of wavefront sets in the content of degenerate Whittaker model, see, e.g. \cite{Mu03}, more recent study \cite{CMBO24} and references therein, and hence our study should also provide some complementary information.





\subsection{Discussions on applications}

For $\pi \in \mathrm{Irr}_{\rho}$ and $\Delta \in \mathrm{Seg}_{\rho}$, instead of studying $D_{\Delta}(\pi)$, one studies on a so-called big derivative in \cite{Ch22+} involving some higher structures. It is shown in \cite{Ch22+} to be useful to study a reduced decomposition \cite{AL23} for $\pi$ in the following sense:
\begin{align} \label{eq reduced decomp}
   \mathrm{St}(\mathfrak p) \times D_{\mathfrak p}(\pi) \twoheadrightarrow \pi \quad (\mbox{equivalently, $\pi \hookrightarrow D_{\mathfrak p}(\pi)\times \mathrm{St}(\mathfrak p)$ }) ,
\end{align}
where $\mathfrak p=\mathfrak{mx}(\pi, \Delta)$ for some segment $\Delta$ (see (\ref{eqn mx segment}) for the definition of $\mathfrak{mxpt}^a$). In Appendix B, we give a generalization to mutlisegment cases. Such reduced decomposition is also useful to study the relation between Bernstein-Zelevinsky derivatives and layers of Bernstein-Zelevinsky filtrations in \cite{Ch22+b}.

Another application is to give an inductive construction of some simple quotients of Bernstein-Zelevinsky derivatives (which is a main goal in this series of articles). For example, for $\pi \in \mathrm{Irr}_{\rho}$ and $\Delta \in \mathrm{Seg}_{\rho}$, if a simple quotient of $\pi^{(i)}$ is $\Delta$-reduced in the sense that $\mathfrak{mx}(\pi, \Delta)=\emptyset$ (see (\ref{eqn mx segment} for a detailed notion), then one may construct such simple quotient from a simple quotient of $(D_{\mathfrak p}(\pi))^{(i-l)}$ via (\ref{eq reduced decomp}), where $l=l_{abs}(\mathfrak p)$. The idea of this construction is closely related to the commutativity discussed above and see Proposition \ref{prop inductive construction} for a precise statement.


\subsection{Organization}

In first few sections, we recall some main ingredients:  highest derivative multisegments in Section \ref{s highest derivative multi}, removal processes in Section \ref{s highest derivative removal}, and non-overlapping and intermediate segment properties in Section \ref{s non-overlapping}.

Sections \ref{s two segment commut} to \ref{s commutative minimal} study the commutativity and subsequent property for minimal sequences. The approach is largely combinatorial using the overlapping property. Section \ref{s two segment commut} studies the two segment case while Section \ref{s three segment cases} studies the three segment case. Section \ref{s prelim subsequent and commut} shows some preliminary results for general cases. Sections \ref{s subseq property} and \ref{s commutative minimal} prove the general case for the  subsequent and commutativity property respectively. 

Sections \ref{s eta invariant and commutativity} to \ref{s embedding minimality} study some representation-theoretic aspects of the minimality. Section \ref{s eta invariant and commutativity} explains a representation-theoretic proof of commutativity of two segment case.  Section \ref{s conj model minimal} conjectures a representation-theoretic interpretation for the minimality and proves for the two segment case. Sections \ref{s application embedding models} and \ref{s embedding minimality} study how the interpretation gives some applications and connections to removal processes.

\subsection{Acknowledgements}
This article is benefited from author's visit to NCTS at Taiwan in December 2023 and December 2024, and the author would like to thank the center for hospitality. This project is partly supported by the Research Grants Council of the Hong Kong Special Administrative Region, China (Project No: 	17305223) and NSFC grant for Excellent Young Scholar (Project No.: 12322120).

\section{Highest derivative multisegments} \label{s highest derivative multi}


The highest derivative multisegment is introduced in \cite[Section 8.1]{Ch22+d} as a main tool for the entire study. In this section, we first recall the definition and then prove a new realization theorem, which is of independent interests.

\subsection{Highest derivative multisegments} \label{ss highest derivative mult}

 For $c \in \mathbb Z$ and $\pi \in \mathrm{Irr}_{\rho}$, define $\mathfrak{mxpt}^a(\pi, c)$ to be the maximal multisegment such that 
\begin{enumerate}
\item for any $\Delta \in \mathfrak{mxpt}^a(\pi,c)$, $a(\Delta)=c$; and
\item $D_{\mathfrak{mxpt}^a(\pi,c)}(\pi)\neq 0$.
\end{enumerate}
Here the maximality is to take the lexicographical ordering on the $b(\Delta)$ values for all segments in $\mathfrak{mxpt}^a(\pi, c)$ . See \cite{Ch22+d} for details and examples. Define the {\it highest derivative multisegment} of $\pi \in \mathrm{Irr}_{\rho}$ to be
\[  \mathfrak{hd}(\pi):=\sum_{c\in \mathbb Z} \mathfrak{mxpt}^a(\pi, c).
\]
It is shown in \cite[Theorem 7.3]{Ch22+d} that $D_{\mathfrak{hd}(\pi)}(\pi)$ is the highest derivative of $\pi$ in the sense of \cite{Ze80}. An effecient algorithm in computing the highest derivative multisegment is given in \cite{CP25}. However, a key point is that effects of derivatives can be reflected from the removal process on the highest derivative multisegment, and so we can directly study removal processes.


\subsection{Realization Theorem}

For $d, m \in \mathbb Z_{> 0}$ and a cuspidal representation $\rho$, define $u_{\rho}(d,m)$ to be the unique simple quotient of 
\[ \mathrm{St}\left( \left[-\frac{d-m}{2}, \frac{d+m-2}{2}\right]_{\rho} \right) \times \mathrm{St}\left( \left[-\frac{d-m}{2}-1, \frac{d+m-2}{2}-1\right]_{\rho} \right)  \times \ldots \times 
\mathrm{St}\left(\left[-\frac{d+m-2}{2}, \frac{d-m}{2} \right]_{\rho} \right) .
\]
In other words, $u_{\rho}(d,m)$ is the Langlands quotient of the above parabolically induced module, see \cite[Section 2.6]{Ch22+d} for some notations. One special property of $u_{\rho}(d,m)$ is that it is unitarizable, while this is not the main property used in this content.

An irreducible representation $\pi$ of $G_n$ is said to be an {\it essentially Speh representation} if $\pi \cong \nu^c \cdot u_{\rho}(d,m)$ for some $c \in \mathbb Z$ and some $d,m \in \mathbb Z_{>0}$. Denote such representation by $u_{\rho}(c,d,m)$. We now show the following realization theorem by an explicit construction. The key idea is to exploit some commutativity of the product of some essentially Speh representations. For results in this theme, one sees e.g. \cite{Ba08}, \cite{Ta15}, \cite{LM16} and \cite{Ch21}.

To facilitate discussions, for $\pi \in \mathrm{Irr}_{\rho}$, define $\mathrm{csupp}(\pi)$ to be the multiset of cuspidal representations $\left\{ \rho_1, \ldots, \rho_r\right\}$ such that $\pi$ is an irreducible composition factor in $\rho_1 \times \ldots \times \rho_r$.

\begin{lemma} \label{lem irreducible product}
Let $u_{\rho}(c,d,m)$ be an essentially Speh representation. Let $\pi \in \mathrm{Mult}_{\rho}$. Suppose, for any $\rho' \in \mathrm{csupp}(\pi)$, $\rho' \in \mathrm{csupp}(u_{\rho}(c,d,m))$. Then $\pi \times u_{\rho}(c,d,m)$  is irreducible, and 
\[   \pi \times u_{\rho}(c,d,m) \cong u_{\rho}(c,d,m) \times \pi .
\]
\end{lemma}

\begin{proof}
 See e.g. \cite[Section 6]{LM16}, \cite[Section 9]{Ch21} and \cite[Proposition 5.2]{Ch22+}.
\end{proof}

The highest derivative multisegments for essentially Speh representations are particularly simple:

\begin{lemma}
Let $u_{\rho}(c,d,m)$ be an essentially Speh representation. Then 
\[  \mathfrak{hd}(u_{\rho}(c,d,m)) =\left\{ [c-\frac{d-m}{2}, c+\frac{d+m-2}{2}]_{\rho} \right\} .
\]
\end{lemma}

\begin{proof}
See \cite[Section 11.3]{Ch22+d}.
\end{proof}

\begin{theorem} \label{thm realize highest derivative mult}
Let $\mathfrak m \in \mathrm{Mult}_{\rho}$. Then there exists $\pi \in \mathrm{Irr}_{\rho}$ such that 
\[  \mathfrak{hd}(\pi)=\mathfrak m . \]
\end{theorem}

\begin{proof}
We label the segments in $\mathfrak m$ as:
\[  \Delta_1, \ldots, \Delta_r \]
such that $b(\Delta_1)\leq b(\Delta_2)\leq \ldots \leq b(\Delta_r)$.

We simply let $\pi_1=\mathrm{St}(\Delta_1)$. It is clear that $\mathfrak{hd}(\pi_1)=\left\{ \Delta_1 \right\}$. Now, for $i \geq 2$, we recursively define $\pi_i$ to be an essentially Speh representation $u_{\rho}(c_i, d_i, m_i)$ such that for any $\sigma \in \mathrm{csupp}(\pi_{i-1})$, $\sigma \in \mathrm{csupp}(\pi_i)$. We just have to justify such $\pi_i$ exists. To see this, we write $\Delta_i=[a_i,b_i]_{\rho}$ and we can first choose $d_i$ large enough such that any representation in $\mathrm{csupp}(\pi_{i-1})$ lies in $[b_i-d_i+1,b_i]_{\rho}$. (Such $d_i$ exists by using $b(\Delta_i) \geq b(\Delta_{i-1})$.) For such fixed $d_i$, now we solve $c_i$ and $m_i$ such that $\mathfrak{hd}(u_{\rho}(c_i,d_i,m_i))=\Delta_i$. 

Now let 
\[  \pi=\pi_1 \times \ldots \times \pi_r .\]
The cuspidal conditions guarantee that $\pi$ is irreducible by applying Lemma \ref{lem irreducible product} multiple times. By  \cite[Proposition 11.1]{Ch22+d}, (also see the Arthur type representation in \cite[Proposition 11.2]{Ch22+d}), we have that 
\[  \mathfrak{hd}(\pi)=\mathfrak{hd}(\pi_1)+\ldots+\mathfrak{hd}(\pi_r)=\mathfrak m .
\]
\end{proof}

\part{Combinatorial aspects}

\section{Removal processes} \label{s highest derivative removal}

In this section, we recall some results of removal processes in \cite{Ch22+d}. 

\subsection{More notations on multisegments} \label{ss notations}


Let $\mathrm{deg}(\rho)$ be the number such that $\rho$ is in $\mathrm{Irr}(G_{\mathrm{deg}(\rho)})$. For a segment $\Delta=[a,b]_{\rho}$, let $l_{abs}(\Delta)=(b-a+1)\mathrm{deg}(\rho)$. For a multisegment $\mathfrak m$ in $\mathrm{Mult}_{\rho}$ and an integer $c$, let $\mathfrak m[c]$ be the submultisegment of $\mathfrak m$ containing all the segments $\Delta$ satisfying $a(\Delta)=c$; and let $\mathfrak m\langle c \rangle$ be the submultisegment of $\mathfrak m$ containing all the segments $\Delta$ satisfying $b(\Delta)=c$. 

For a multisegment $\mathfrak m=\left\{ \Delta_1, \ldots, \Delta_k \right\}$, we also set:
\[  l_{abs}(\mathfrak m)=l_{abs}(\Delta_1)+\ldots+l_{abs}(\Delta_k) .
\]




\subsection{Removal process} \label{ss removal process}

We write $[a,b]_{\rho} \prec^L [a',b']_{\rho}$ if either $a<a'$; or $a=a'$ and $b<b'$. A segment $\Delta=[a,b]_{\rho}$ is said to be {\it admissible} to a multisegment $\mathfrak h$ if there exists a segment of the form $[a,c]_{\rho}$ in $\mathfrak h$ for some $c\geq b$.  We now recall the {\it removal process}.

\begin{definition} \cite[Section 8.2]{Ch22+d} \label{def removal process}
Let $\mathfrak h\in \mathrm{Mult}_{\rho}$ and let $\Delta=[a,b]_{\rho}$ be admissible to $\mathfrak h$. The removal process on $\mathfrak h$ by $\Delta$ is an algorithm to carry out the following steps:
\begin{enumerate}
\item Pick a shortest segment $[a,c]_{\rho}$ in $\mathfrak h[a]$ satisfying $b \leq c$. Set $\Delta_1=[a,c]_{\rho}$. Set $a_1=a$ and $b_1=c$.
\item One recursively finds the $\prec^L$-minimal segment $\Delta_i=[a_i,b_i]_{\rho}$ in $\mathfrak h$ such that $a_{i-1}<a_i$ and $b\leq b_i<b_{i-1}$ \footnote{There is a missing condition in \cite{Ch22+d} and \cite{Ch22+e}: $b\leq b_i$. The author would like to thank Peng Zhou for pointing out that. Alternatively, one can use Lemma \ref{lem removal process}(1) to compute the removal process. There are no changes in results and proofs in \cite{Ch22+d} and \cite{Ch22+e}. }. The process stops if one can no longer find those segments.
\item Let $\Delta_1, \ldots, \Delta_r$ be all those segments. For $1\leq i<r$, define $\Delta_i^{tr}=[a_{i+1}, b_i]_{\rho}$ and $\Delta_r^{tr}=[b_{r}+1,b]_{\rho}$ (possibly empty).
\item Define
\[  \mathfrak r(\Delta, \mathfrak h):= \mathfrak h-\sum_{i=1}^r\Delta_i+\sum_{i=1}^r \Delta_i^{tr} .
\]
\end{enumerate} 
We call $\Delta_1, \ldots, \Delta_r$ to be the {\it removal sequence} for $(\Delta, \mathfrak h)$. We also define $\Upsilon(\Delta, \mathfrak h)=\Delta_1$, the first segment of the removal sequence. If $\Delta$ is not admissible to $\mathfrak h$, we set $\mathfrak r(\Delta, \mathfrak h)=\infty$, called the infinity multisegment. We also set $\mathfrak r(\Delta, \infty)=\infty$. 
\end{definition}

\subsection{Properties of removal process}
For a segment $\Delta=[a,b]_{\rho} \neq \emptyset$, let ${}^-\Delta=[a+1,b]_{\rho}$. Two non-empty segments $\Delta=[a,b]_{\rho}$ and $\Delta'=[a',b']_{\rho}$ in $\mathrm{Seg}_{\rho}$ are said to be {\it linked} if one of the following conditions holds:
\begin{enumerate}
\item $a <a' \leq b+1\leq b' $; or
\item $a'<a \leq b'+1 \leq b$.
\end{enumerate}
 Otherwise, $\Delta$ and $\Delta'$ are called to be not linked or unlinked. For two linked segments $\Delta, \Delta'$, we write $\Delta <\Delta'$ if it is in the first condition above. Otherwise, we write $\Delta' <\Delta$. For example, $[2,3]_{\rho}<[4,5]_{\rho}$. 
 
 We recall the following properties for computations:

\begin{lemma} \label{lem removal process} \cite{Ch22+d}
Let $\mathfrak h \in \mathrm{Mult}_{\rho}$ and let $\Delta, \Delta' \in \mathrm{Seg}_{\rho}$ be admissible to $\mathfrak h$. Then
\begin{enumerate}
\item \cite[Lemma 8.7]{Ch22+d} Let $\mathfrak h^*=\mathfrak h-\Upsilon(\Delta, \mathfrak h)+{}^-\Upsilon(\Delta, \mathfrak h)$. Then $\mathfrak r(\Delta, \mathfrak h)=\mathfrak r({}^-\Delta, \mathfrak h^*)$.
\item \cite[Lemma 8.8]{Ch22+d} Write $\Delta=[a,b]_{\rho}$. For any $a'<a$, $\mathfrak r(\Delta, \mathfrak h)[a']=\mathfrak h[a']$.
\item  \cite[Lemma 8.9]{Ch22+d} If $\Delta \in \mathfrak h$, then $\mathfrak r(\Delta, \mathfrak h)=\mathfrak h-\Delta$.
\item \cite[Lemma 8.10]{Ch22+d} Suppose $a(\Delta)=a(\Delta')$. Then 
\[  \Upsilon(\Delta, \mathfrak h)+\Upsilon(\Delta', \mathfrak r(\Delta, \mathfrak h))=\Upsilon(\Delta', \mathfrak h)+\Upsilon(\Delta, \mathfrak r(\Delta', \mathfrak h)) .
\]
\item \cite[Lemma 8.12]{Ch22+d} If $\Delta, \Delta'$ are unlinked, then $\mathfrak r(\Delta', \mathfrak r(\Delta, \mathfrak h))=\mathfrak r(\Delta, \mathfrak r(\Delta', \mathfrak h))$.
\end{enumerate}
\end{lemma}

For $\mathfrak h \in \mathrm{Mult}_{\rho}$, a multisegment $\mathfrak n=\left\{ \Delta_1, \ldots, \Delta_r\right\} \in \mathrm{Mult}_{\rho}$ written in an ascending order, define:
\[   \mathfrak r(\mathfrak n, \mathfrak h)= \mathfrak r(\Delta_r, \ldots , \mathfrak r(\Delta_1, \mathfrak h)\ldots ) .  \]
We say that $\mathfrak n$ is {\it admissible} to $\mathfrak h$ if $\mathfrak r(\mathfrak n, \mathfrak h)\neq \infty$.

\begin{theorem} \cite[Theorem 10.2]{Ch22+d} \label{thm isomorphic derivatives} 
Let $\pi \in \mathrm{Irr}_{\rho}$. Let $\mathfrak m, \mathfrak m' \in \mathrm{Mult}_{\rho}$ be admissible to $\pi$. Then $\mathfrak m, \mathfrak m'$ are admissible to $\mathfrak{hd}(\pi)$, and furthermore, $D_{\mathfrak m}(\pi)\cong D_{\mathfrak m'}(\pi)$ if and only if $\mathfrak r(\mathfrak m, \pi)=\mathfrak r(\mathfrak m', \pi)$.
\end{theorem}

\subsection{More relations to derivatives}

 For $\mathfrak h \in \mathrm{Mult}_{\rho}$ and $\Delta=[a,b]_{\rho} \in \mathrm{Seg}_{\rho}$, set
\begin{align} \label{eqn varepsilon combinatorics}
\varepsilon_{\Delta}(\mathfrak h)=|\left\{ \widetilde{\Delta} \in \mathfrak h[a]: \Delta\subset \widetilde{\Delta} \right\}| . 
\end{align}
Define $\varepsilon_{\Delta}(\pi):=\varepsilon_{\Delta}(\mathfrak{hd}(\pi))$, which is equivalent to a different formulation of the same notation in \cite[Section 4.2]{Ch22+d}. We also remark that when $\Delta$ is a singleton, this number coincides with the number defined in \cite[Definition 2.1.1]{Ja07}. In terms of derivatives, we have the following:

\begin{theorem} \cite[Theorem 9.3]{Ch22+d} \label{thm effect of Steinberg}
Let $\pi \in \mathrm{Irr}_{\rho}$. Let $\Delta=[a,b]_{\rho} \in \mathrm{Seg}_{\rho}$ be admissible to $\pi$. Let $\Delta'=[a',b']_{\rho}$ be another segment in $\mathrm{Seg}_{\rho}$. If either $a'>a$; or $\Delta'$ and $\Delta$ are unlinked, then 
\[ \varepsilon_{\Delta'}(D_{\Delta}(\pi))=\varepsilon_{\Delta'}(\mathfrak r(\Delta, \pi)). \]
\end{theorem}

\begin{theorem} \cite[Theorem 9.3]{Ch22+d} \label{thm effect of Steinberg small a}
Let $\pi \in \mathrm{Irr}_{\rho}$. Let $\Delta=[a,b]_{\rho} \in \mathrm{Seg}_{\rho}$ be admissible to $\pi$. Let $\Delta'=[a',b']_{\rho}$ be another segment. If $a' <a$, then 
\[  \varepsilon_{\Delta'}(D_{\Delta}(\pi)) \geq \varepsilon_{\Delta'}(\mathfrak r(\Delta, \pi))=\varepsilon_{\Delta'}(\pi) .
\]

\end{theorem}

We remark that the equality in Theorem \ref{thm effect of Steinberg small a} follows from Lemma \ref{lem removal process}(2).

\section{Non-overlapping property for a sequence} \label{s non-overlapping}

In \cite{Ch22+e}, we have shown some characterizations for the minimality of two segment case. The goal of this section is to generalize a so-called non-overlapping property to a multisegment case. For this, we first recall some ingredients in \cite{Ch22+e}: fine chains and local minimizability. Those ingredients are combinatorial in nature and so most statements will be formulated for $\mathrm{Mult}_{\rho}$ rather than $\mathrm{Irr}_{\rho}$.

\subsection{Zelevinsky ordering} \label{ss zel ordering}

For two linked segments $\Delta=[a,b]_{\rho}$ and $\Delta'=[a',b']_{\rho}$ with $\Delta<\Delta'$, we define:
\[ \Delta \cup \Delta' = [a,b']_{\rho}, \quad \Delta \cap \Delta'=[a',b]_{\rho} .
\]
A multisegment $\mathfrak n$ is said to be obtained from another multisegment $\mathfrak m$ by an {\it elementary intersection-union process} if there exists a pair of linked segments $\Delta, \Delta'$ in $\mathfrak m$ such that
\[  \mathfrak n=\mathfrak m-\Delta-\Delta'+\Delta\cup \Delta'+\Delta\cap \Delta' .
\]
Here the subtractions mean the (multi-)set theoretic subtractions and additions mean the (multi-)set theoretic unions. We shall also use such notions for subtractions and additions later. Note that $\Delta\cap \Delta'$ is possibly the empty set and in such case, we simply drop the term. For example, 
$\left\{ [1,3]_{\rho}, [2]_{\rho}, [4,5]_{\rho} \right\}$ and $\left\{ [1,2]_{\rho}, [2,5]_{\rho} \right\}$ are obtained from $\left\{ [1,2]_{\rho}, [2,3]_{\rho}, [4,5]_{\rho} \right\}$ by elementary intersection-union processes.

For two multisegments $\mathfrak m, \mathfrak n \in \mathrm{Mult}_{\rho}$, we wrtie $\mathfrak n \leq_Z \mathfrak m$ if $\mathfrak n$ is obtained from $\mathfrak m$ by a sequence of intersection-union processes, or $\mathfrak m=\mathfrak n$. It is well-known from \cite{Ze80} that $\leq_Z$ defines a partial ordering on $\mathrm{Mult}_{\rho}$. 

\subsection{Non-overlapping property and intermediate segment property}

We first define the $\eta$-invariant, which is crucial in studying the minimality of sequences of derivatives:
\begin{definition} \label{def eta invariant}
\begin{itemize}
\item[(1)] Let $\mathfrak h \in \mathrm{Mult}_{\rho}$. Let $\Delta=[a,b]_{\rho} \in \mathrm{Seg}_{\rho}$ be admissible to $\mathfrak h$. Define
\[  \eta_{\Delta}(\mathfrak h) := (\varepsilon_{[a,b]_{\rho}}(\mathfrak h), \varepsilon_{[a+1,b]_{\rho}}(\mathfrak h), \ldots, \varepsilon_{[b,b]_{\rho}}(\mathfrak h)) .
\]
\item[(2)] Let $\pi \in \mathrm{Irr}_{\rho}$. Let $\Delta \in \mathrm{Seg}_{\rho}$ be admissible to $\pi$. Define $\eta_{\Delta}(\pi)=\eta_{\Delta}(\mathfrak{hd}(\pi))$. 
\end{itemize}
\end{definition}

\begin{definition} \label{def overlapping property}
Let $\Delta, \Delta' \in \mathrm{Seg}_{\rho}$ with $\Delta < \Delta'$. Let $\mathfrak h \in \mathrm{Mult}_{\rho}$. Suppose $\Delta$ is admissible to $\mathfrak h$. 
\begin{itemize}
\item[(1)] We say that $(\Delta, \Delta', \mathfrak h)$ satisfies the {\it non-overlapping property} if 
\[ \eta_{\Delta'}(D_{\Delta}(\pi))=\eta_{\Delta'}(\pi). \] 
\item[(2)] We say that $(\Delta, \Delta', \mathfrak h)$ satisfies the {\it intermediate segment property} if there exists a segment $\widetilde{\Delta} \in \mathfrak h$ such that 
\[   a(\Delta) \leq a(\widetilde{\Delta}) <a(\Delta') , \mbox{ and } a(\widetilde{\Delta})\leq b(\Delta) \leq b(\widetilde{\Delta}) <b(\Delta') . \]
\end{itemize}
\end{definition}

We remark that the original formulation of the non-overlapping property in \cite{Ch22+e} is phrased in terms of some properties among segments related to intersections between segments. We shall use the above equivalent combinatorial formulation (see \cite[Proposition 9.5]{Ch22+e}), which will be more useful for our study later.

\begin{proposition}\cite[Proposition 9.5]{Ch22+e} \label{prop nonoverlapping property}
Let $\Delta, \Delta', \mathfrak h$ be as in Definition \ref{def overlapping property}. We further assume that $\Delta'$ is admissible to $\mathfrak h$. Then the following conditions are equivalent:
\begin{itemize}
\item[(1)] $(\Delta, \Delta', \mathfrak h)$ satisfies the non-overlapping property.
\item[(2)] $(\Delta, \Delta', \mathfrak h)$ satisfies the intermediate segment property.
\item[(3)] $\left\{ \Delta, \Delta' \right\}$ is the $\leq_Z$-minimal element in $\mathcal S(\mathfrak h, \mathfrak r(\left\{ \Delta, \Delta'\right\}, \mathfrak h))$, where $\mathcal S(\mathfrak h, \mathfrak r(\left\{ \Delta, \Delta'\right\}, \mathfrak h)$ is defined as:
\[  \mathcal S(\mathfrak h, \mathfrak r(\left\{ \Delta, \Delta'\right\}, \mathfrak h):= \left\{ \mathfrak n \in \mathrm{Mult}_{\rho}: \mathfrak r(\mathfrak n, \mathfrak h)=\mathfrak r(\left\{ \Delta, \Delta'\right\}, \mathfrak h)\right\} .
\]
\end{itemize}
\end{proposition}

\subsection{Fine chains}

\begin{definition} \label{def truncation}
Let $\mathfrak h \in \mathrm{Mult}_{\rho}$ and let $\mathfrak n \in \mathrm{Mult}_{\rho}$. Let $a$ be the smallest integer such that $\mathfrak n[a] \neq \emptyset$. Write $\mathfrak n[a]=\left\{ \Delta_1, \ldots, \Delta_k \right\}$.  
\begin{itemize}
\item Define $\mathfrak r_1=\mathfrak h$. For $i \geq 2$, define
\[  \mathfrak r_i:= \mathfrak r(\Delta_{i-1}, \ldots , \mathfrak r(\Delta_1, \mathfrak h)\ldots) .
\]
Define 
\[  \mathfrak{fs}(\mathfrak n, \mathfrak h):= \left\{ \Upsilon(\Delta_1, \mathfrak r_1), \ldots , \Upsilon(\Delta_k, \mathfrak r_k) \right\} . \]
\item Define
\[  \mathfrak{trr}(\mathfrak n, \mathfrak h):= \mathfrak h-\mathfrak{fs}(\mathfrak n, \mathfrak h)+{}^-(\mathfrak{fs}(\mathfrak n, \mathfrak h))
\]
and
\[  \mathfrak{trd}(\mathfrak n, \mathfrak h):= \mathfrak n- \mathfrak n[a]+{}^-(\mathfrak n[a]) .
\]
\end{itemize}
\end{definition}

\begin{definition} \label{def fine chains and sequences}
Let $\mathfrak h \in \mathrm{Mult}_{\rho}$. Let $\mathfrak n \in \mathrm{Mult}_{\rho}$ be admissible to $\mathfrak h$. Set $\mathfrak h_0=\mathfrak h$ and $\mathfrak n_0=\mathfrak n$. Define recursively 
\[  \mathfrak h_i=\mathfrak{trr}(\mathfrak n_{i-1}, \mathfrak h_{i-1}) , \quad \mathfrak n_i=\mathfrak{trd}(\mathfrak n_{i-1}, \mathfrak h_{i-1}) .
\]
The {\it fine chain for the removal process} for $(\mathfrak n, \mathfrak h)$ (or simply fine chain for $(\mathfrak n, \mathfrak h)$) is the sequence 
\[ \mathfrak{fs}(\mathfrak n_0, \mathfrak h_0), \mathfrak{fs}(\mathfrak n_1, \mathfrak h_1), \ldots 
\]
The sequences $\mathfrak h_0, \mathfrak h_1, \ldots$ and $\mathfrak n_0, \mathfrak n_1, \ldots$ will also be useful later.

\end{definition}

\begin{lemma} \label{lem reomval equaltiy fine chain}
We use the notations in Definition \ref{def fine chains and sequences}. Then, for all $i$,
\[  \mathfrak r(\mathfrak n, \mathfrak h)=\mathfrak r(\mathfrak n_i, \mathfrak h_i) .
\]
\end{lemma}

\begin{proof}
This follows from repeated uses of Lemma \ref{lem removal process}(1), or a slightly stronger statement of \cite[Lemma 3.4]{Ch22+e}. 
\end{proof}

\subsection{Local minimizability}

\begin{definition} \label{def local min}
Let $\mathfrak h \in \mathrm{Mult}_{\rho}$ and let $\mathfrak n \in \mathrm{Mult}_{\rho}$. Let $a$ be the smallest integer such that $\mathfrak n[a]\neq \emptyset$. We say that $(\mathfrak n, \mathfrak h)$ is {\it locally minimizable} if there exists a segment $\Delta$ in $\mathfrak n[a+1]$ such that 
\[   |\left\{ \widetilde{\Delta} \in \mathfrak n[a] : \Delta \subset \widetilde{\Delta} \right\}| < |\left\{ \widetilde{\Delta} \in \mathfrak{fs}(\mathfrak n, \mathfrak h): \Delta \subset \widetilde{\Delta} \right\} |.
\]
\end{definition}
Note that each $\widetilde{\Delta} \in \mathfrak n[a]$ satisfying $\Delta \subset \widetilde{\Delta}$ guarantees the corresponding first segment $\overline{\Delta}$ in the removal sequence satisfying $\Delta \subset \overline{\Delta}$. Roughly speaking, when the difference of two cardinalities in Definition \ref{def local min} is non-zero, one can find a "short" segment in $\mathfrak n[a]$ to do the intersection-union process which still does not change the choices of first segments in the removal processes.

\begin{theorem} \cite[Lemma 6.1 and Lemma 6.2]{Ch22+e} \label{thm minimizability}
Let $\mathfrak h \in \mathrm{Mult}_{\rho}$. Let $\mathfrak n \in \mathrm{Mult}_{\rho}$ be admissible to $\mathfrak h$. Let $\mathfrak h_0, \mathfrak h_1, \ldots$ and $\mathfrak n_0, \mathfrak n_1, \ldots$ be as constructed in Definition \ref{def fine chains and sequences}. Then $\mathfrak n$ is minimal to $\mathfrak h$ if and only if $(\mathfrak n_i, \mathfrak h_i)$ is not locally minimizable for all $i$.
\end{theorem}

\subsection{Non-overlapping property for a sequence}

We now generalize Proposition \ref{prop nonoverlapping property} to a multisegment situation, which will be used in Section \ref{ss case 13 minimal to pi}. We first prove a lemma:

\begin{lemma} \label{lem a computation on eta change}
Let $\Delta=[a,b]_{\rho}$ and $\Delta'=[c,d]_{\rho}$ be two segments. Suppose $\Delta'<\Delta$. Let $\mathfrak h\in \mathrm{Mult}_{\rho}$ such that $\Delta'$ is admissible to $\mathfrak h$. If $\Delta \subset \Upsilon(\Delta', \mathfrak h)$, then $\eta_{\Delta}(\mathfrak r(\Delta', \mathfrak h))\neq \eta_{\Delta}(\mathfrak h)$. 
\end{lemma}

\begin{proof}
Let $\Delta_1, \ldots, \Delta_r$ be the removal sequence for $(\Delta', \mathfrak h)$. Let $i^*$, which exists from our assumption, be the largest integer such that $b(\Delta) \leq b(\Delta_{i^*})$. Write $\Delta^{tr}_{i^*}=[a_{i^*},b_{i^*}]_{\rho}$. Then $\Delta_{i^*}^{tr}\neq \emptyset$ contributes extra one to $\varepsilon_{[a_{i^*},b]_{\rho}}$ for $\mathfrak r(\Delta', \mathfrak h)$. However, the segments $\Delta_1, \ldots, \Delta_r$ as well as $\Delta^{tr}_j$ $(j\neq i^*)$ do not contribute to $\varepsilon_{[a_{i^*},b]_{\rho}}$. This implies the desired inequality. 
\end{proof}

\begin{proposition} \label{prop dagger property}
Let $\mathfrak h \in \mathrm{Mult}_{\rho}$. Let $\mathfrak n \in \mathrm{Mult}_{\rho}$ be minimal to $\mathfrak h$. Let $\Delta=[a,b]_{\rho}$ be a segment such that $a(\Delta')<a$ and $b(\Delta')<b$ for any $\Delta' \in \mathfrak n$. Then $\mathfrak n+\Delta$ is still minimal to $\mathfrak h$ if and only if 
\[ \eta_{\Delta}(\mathfrak r(\mathfrak n, \mathfrak h))=\eta_{\Delta}(\mathfrak h). \]
\end{proposition}

\begin{proof}
We construct a sequence of multisegments $\mathfrak n_0, \mathfrak n_1, \mathfrak n_2, \ldots$ and $\mathfrak h_0, \mathfrak h_1, \mathfrak h_2, \ldots$ as in Definition \ref{def fine chains and sequences}. By Lemma \ref{lem reomval equaltiy fine chain}, we have:
\[ (*)\quad  \mathfrak r(\mathfrak n_0, \mathfrak h_0)=\mathfrak r(\mathfrak n_1, \mathfrak h_1)= \ldots
\]
Let $a_i$ be the smallest integer such that $\mathfrak n_i[a_i]\neq \emptyset$. 

 Let $i^*$ be the index such that $a_{i^*}=c-1$. If such index does not exist, it implies that $b(\widetilde{\Delta})<c-1$ for all $\widetilde{\Delta} \in \mathfrak n$. In such case, by a direct computation on removal process using Definition \ref{def removal process}(3) and (4), one has $\mathfrak r(\mathfrak n, \mathfrak h)[x] =\mathfrak h[x]$ for $x \geq c$. In particular, $\eta_{\Delta}(\mathfrak r(\mathfrak n, \mathfrak h))=\eta_{\Delta}(\mathfrak h)$; and $\mathfrak n+\Delta$ is still minimal since $\Delta$ is unlinked to any segment in $\mathfrak n$. In other words, both conditions are automatically satisfied if such $i^*$ does not exist.

We now assume such $i^*$ exists. We first prove the only if direction. By the minimality condition and Theorem \ref{thm minimizability}, $(\mathfrak n_{i^*}+\Delta, \mathfrak h_{i^*})$ is not locally minimizable. On the other hand, the hypothesis in this proposition guarantees that any segment $\widetilde{\Delta}$ in $\mathfrak n_{i^*}$ satisfies $\Delta \not\in \widetilde{\Delta}$ and so
\[  |\left\{ \widetilde{\Delta} \in \mathfrak n_{i^*}[a] : \Delta \subset \widetilde{\Delta} \right\}|=0 .
\]
The local non-minimizability on $(\mathfrak n_{i^*}+\Delta, \mathfrak h_{i^*})$ implies that 
\[  |\left\{ \widetilde{\Delta} \in \mathfrak{fs}(\mathfrak n_{i^*}+\Delta, \mathfrak h_{i^*} ):  \Delta \subset \widetilde{\Delta} \right\}|=0 .
\]
In other words, for any segment $[x,y]_{\rho}$ in $\mathfrak{fs}(\mathfrak n_{i^*}+\Delta, \mathfrak h_{i^*})=\mathfrak{fs}(\mathfrak n_{i^*}, \mathfrak h_{i^*})$ (the equaltiy follows from Definition \ref{def fine chains and sequences}), it must take the form $[x',y']_{\rho}$ for $y'\leq b-1$. Thus the segments involved or produced in the removal process (by the nesting property) cannot contribute to $\eta_{\Delta}(\mathfrak r(\mathfrak n_{i^*}, \mathfrak h_{i^*}))$ and so 
\[ (**) \quad \eta_{\Delta}(\mathfrak r(\mathfrak n_{i^*}, \mathfrak h_{i^*}))=\eta_{\Delta}(\mathfrak h_{i^*}) .
\]

Moreover, $\mathfrak h_{i^*}$ is obtained by truncating left points for $\nu^x\rho$ for some $x\leq c-2$. Thus, $\eta_{\Delta}(\mathfrak h_0)=\ldots =\eta_{\Delta}(\mathfrak h_{i^*})$ by Definitions \ref{def truncation} and \ref{def fine chains and sequences}. Combining with (*) and (**), we have the desired equation, proving the only if direction.

We now prove the if direction. By the construction of $\mathfrak h_i$ (Definition \ref{def fine chains and sequences}), we have:
\[ \eta_{\Delta}(\mathfrak h)=\eta_{\Delta}(\mathfrak h_{i^*}) \]
Now with (*) and the hypothesis $\eta_{\Delta}(\mathfrak h)=\eta_{\Delta}(\mathfrak r(\mathfrak n, \mathfrak h))$, we again have that 
\[  \eta_{\Delta}(\mathfrak r(\mathfrak n_{i^*}, \mathfrak h_{i^*})) =\eta_{\Delta}(\mathfrak h_{i^*}) .
\]
This condition says that $\mathfrak{fs}(\mathfrak n_{i^*}, \mathfrak h_{i^*})$ cannot take the form $[c-1, y]_{\rho}$ for some $y \geq b$, by Lemma \ref{lem a computation on eta change} and the assumption that any segment in $\mathfrak n$ (and so $\mathfrak n_{i^*}$) satisfies $b(\Delta)<b$. In other words, $(\mathfrak n_{i^*}, \mathfrak h_{i^*})$ is not locally minimizable. The local  non-minimizability for other pairs $(\mathfrak n_{j}, \mathfrak h_j)$ for $j<i^*$ follows from the minimality of $\mathfrak n$. Then the minimality of $\mathfrak n+\Delta$ to $\mathfrak h$ follows from Theorem \ref{thm minimizability}.
\end{proof}

\section{Two segment basic case (commutativity)} \label{s two segment commut}

\subsection{Lemma for unlinked segments}
We recall the following first commutativity (see e.g. \cite[Lemma 4.4]{Ch22+d}), which will be used later:

\begin{lemma} \label{lem comm derivative 1}
Let $\Delta, \Delta' \in \mathrm{Seg}_{\rho}$ be unlinked. For any $\pi \in \mathrm{Irr}_{\rho}$, 
\[  D_{\Delta'}\circ D_{\Delta}(\pi)\cong D_{\Delta}\circ D_{\Delta'}(\pi).  \]
\end{lemma}

\subsection{Intermediate segment property under a derivative}

\begin{lemma} \label{lem dagger property on derivative}
Let $\pi \in \mathrm{Irr}_{\rho}$. Let $\Delta, \Delta''\in \mathrm{Seg}_{\rho}$ be admissible to $\pi$. Let $\Delta' \in \mathrm{Seg}_{\rho}$. Suppose $(\Delta, \Delta', \mathfrak{hd}(\pi))$ satisfies the non-overlapping property or intermediate segment property. If $\Delta'' \subset \Delta'$, then $(\Delta, \Delta', \mathfrak{hd}(D_{\Delta''}(\pi)))$ also satisfies the non-overlapping property and the intermediate segment property.
\end{lemma}

\begin{proof}
Since the non-overlapping property and the intermediate segment property are equivalent by Proposition \ref{prop nonoverlapping property}, it suffices to see that $(\Delta, \Delta', \mathfrak{hd}(D_{\Delta'}(\pi)))$ also satisfies the intermediate segment property.


Write $\Delta=[a,b]_{\rho}$, $\Delta'=[a',b']_{\rho}$ and $\Delta''=[a'',b'']_{\rho}$. By Theorems \ref{thm effect of Steinberg} and \ref{thm effect of Steinberg small a}, we have that for any segment $\widetilde{\Delta}=[\widetilde{a}, \widetilde{b}]_{\rho}$, 
\[  (*)\quad   \varepsilon_{\widetilde{\Delta}}(\mathfrak r(\Delta'', \pi))  \leq \varepsilon_{\widetilde{\Delta}}(D_{\Delta''}(\pi))  .\]
From this, one can recover $\mathfrak{hd}(\pi)[c]$ and $\mathfrak r(\Delta, \pi)[c]$ for each integer $c$. The missing part of (*) between $\mathfrak r(\Delta', \mathfrak{hd}(\pi))$ and $ \mathfrak{hd}(D_{\Delta''}(\pi))$ is on some values $a''-1, \ldots, b''-1<b'$. Thus, one obtains $\mathfrak{hd}(D_{\Delta''}(\pi))$ by prolonging some segments in $\mathfrak r(\Delta'', \mathfrak{hd}(\pi))$ using (possibly some of) $a''-1, \ldots, b''-1$. 
 
Now, by the intermediate segment property for $(\Delta, \Delta', \mathfrak{hd}(\pi))$, there exists a segment of the form $[c,d]_{\rho}$ satisfying:
\[   a \leq c <a', \quad b\leq d< b' .
\]
Now, by the above process of obtaining $\mathfrak{hd}(D_{\Delta''}(\pi))$, the segment $[c,d]_{\rho}$ can be prolonged to the form $[c,e]_{\rho}$ in $\mathfrak{r}(\Delta'', \pi)$ for some $e \leq b'-1$. Then the segment $[c,e]_{\rho}$ gives the desired requirement for the intermediate segment property for $(\Delta, \Delta', D_{\Delta''}(\pi))$. Thus, by Proposition \ref{prop nonoverlapping property}, the triple satisfies the two properties.
\end{proof}

For a segment $\Delta=[a,b]_{\rho} \in \mathrm{Seg}_{\rho}$ and $\pi \in \mathrm{Irr}_{\rho}$, define:
\[  |\eta|_{\Delta}(\pi)= \varepsilon_{[a,b]_{\rho}}(\pi)+\varepsilon_{[a+1,b]_{\rho}}(\pi)+\ldots +\varepsilon_{[b,b]_{\rho}}(\pi) . \]
In view of (\ref{eqn varepsilon combinatorics}), $|\eta|_{\Delta}(\pi)$ measures the number of segments $\overline{\Delta}$ in $\mathfrak{hd}(\pi)$ satisfying that $\nu^a\rho\leq a(\overline{\Delta}) \leq \nu^b\rho$ and $\nu^b\rho \leq b(\overline{\Delta})$. For example, suppose we have $\pi \in \mathrm{Irr}_{\rho}$ such that 
\[  \mathfrak{hd}(\pi)= \left\{ [1,4]_{\rho}, [1,3]_{\rho} , [1,2]_{\rho}, [2,5]_{\rho}, [2,4]_{\rho}   \right\}.
\]
Then $|\eta|_{[1,3]_{\rho}}(\pi)=4$, which is contributed from $[1,4]_{\rho}, [1,3]_{\rho}$ for $\varepsilon_{[1,3]_{\rho}}(\pi)$ and from $[2,5]_{\rho}, [2,4]_{\rho}$ for $\varepsilon_{[2,3]_{\rho}}(\pi)$.

\begin{lemma} \label{lem eta for intermediate}
Let $\pi \in \mathrm{Irr}_{\rho}$. Let $\Delta, \Delta' \in \mathrm{Seg}_{\rho}$. Suppose $\Delta$ is admissible to $\pi$. Suppose $(\Delta, \Delta', \mathfrak{hd}(\pi))$ satisfies the non-overlapping property or the intermediate segment property. Let $\widetilde{\Delta}=\Delta \cup \Delta'$.
\[  |\eta|_{\widetilde{\Delta}}(\pi)-|\eta|_{\Delta'}(\pi)=|\eta|_{\widetilde{\Delta}}(D_{\Delta}(\pi))-|\eta|_{\Delta'}(D_{\Delta}(\pi)) .
\]
\end{lemma}

\begin{proof}
Write the removal sequence for $(\Delta, \pi)$ to be:
\[   \Delta_1, \ldots, \Delta_r .\]
By the intermediate segment property, there exists a segment $\Delta_k$ such that $a(\Delta_k)<a(\Delta')$ and $b(\Delta_k)< b(\Delta')$. Let $k^*$ be the smallest such integer. Then, with the nesting property, only $\Delta_1, \ldots, \Delta_{k^*-1}$ (among $\Delta_1, \ldots, \Delta_r$) contribute to $|\eta|_{\widetilde{\Delta}}(\pi)-|\eta|_{\Delta'}(\pi)$. 

Write $\Delta=[a,b]_{\rho}$ and $\Delta'=[a',b']_{\rho}$. By Theorem \ref{thm effect of Steinberg}, for $a \leq c \leq a'-1$
\[   \varepsilon_{[c,b']_{\rho}}(D_{\Delta}(\pi)) =\varepsilon_{[c,b']_{\rho}}(\mathfrak r(\Delta, \pi)) .\]
Then, only $\Delta_1^{tr}, \ldots, \Delta_{k^*-1}^{tr}$ (among $\Delta_1^{tr}, \ldots, \Delta_{r}^{tr}$) can contribute to $|\eta|_{\widetilde{\Delta}}(\pi)-|\eta|_{\Delta'}(\pi)$ and so this implies the equality.
\end{proof}

\subsection{Commutativity and minimality for two segment case}


\begin{proposition} \label{prop dagger property 2}
Let $\pi \in \mathrm{Irr}_{\rho}$. Let $\Delta_1, \Delta_2 \in \mathrm{Seg}_{\rho}$ be admissible to $\pi$. Suppose $\Delta_1<\Delta_2$ and $(\Delta_1, \Delta_2, \mathfrak{hd}(\pi))$ satisfies the non-overlapping property, or equivalently 
\[  D_{\Delta_2}\circ D_{\Delta_1}(\pi) \not\cong D_{\Delta_1\cup \Delta_2}\circ D_{\Delta_1\cap \Delta_2}(\pi) .
\]
 Then 
\[  D_{\Delta_2}\circ D_{\Delta_1}(\pi) \cong D_{\Delta_1}\circ D_{\Delta_2}(\pi).
\]

\end{proposition}

\begin{proof}

The equivalence of the two conditions follows from Proposition \ref{prop nonoverlapping property}. 

 We shall use the notations in the proof. Indeed, in view of a criteria of commutativity in \cite[Proposition 6.2]{Ch22+d}, it suffices to prove 
\[   D_{\Delta_1}\circ D_{\Delta_2}(\pi) \not\cong D_{\Delta_1\cup \Delta_2}\circ D_{\Delta_1\cap \Delta_2}(\pi) .\]
Let $\widetilde{\Delta}=\Delta_1\cup \Delta_2$. To this end, it suffices to show that 
\begin{align} \label{eqn absolute eta diff}
  (|\eta|_{\widetilde{\Delta}}-|\eta|_{\Delta_2})(D_{\Delta_1}\circ D_{\Delta_2}(\pi)) =(|\eta|_{\widetilde{\Delta}}-|\eta|_{\Delta_2})(\pi) .
\end{align}

Note that, by the unlinked part of Theorem \ref{thm effect of Steinberg},
\[ (|\eta|_{\widetilde{\Delta}}-|\eta|_{\Delta_2})( D_{\Delta_2}(\pi))=(|\eta|_{\widetilde{\Delta}}-|\eta|_{\Delta_2})(\pi) \] 

On the other hand, by Proposition \ref{prop nonoverlapping property}(2), $(\Delta_1, \Delta_2, \pi)$ satisfies the intermediate segment property and so $(\Delta_1, \Delta_2, D_{\Delta_2}(\pi))$ also satisfies the intermediate segment property by Lemma \ref{lem dagger property on derivative}. Now, by Lemma \ref{lem eta for intermediate},

\[  (|\eta|_{\widetilde{\Delta}}-|\eta|_{\Delta_2})(D_{\Delta_1}\circ D_{\Delta_2}(\pi))=(|\eta|_{\widetilde{\Delta}}-|\eta|_{\Delta_2})(D_{\Delta_2}(\pi)) .\]
Combining above two equations, we have (\ref{eqn absolute eta diff}) as desired.
\end{proof}

We shall give a proof of Proposition \ref{prop dagger property 2} from a representation-theoretic perspective (see Lemma \ref{lem commut for max case}).

\section{Some preliminary results for subsequent and commutativity properties} \label{s prelim subsequent and commut}

\subsection{Cancellative property}

\begin{proposition} (Cancellative property) \label{prop cancel property}
Let $\pi \in \mathrm{Irr}_{\rho}$. Let $\mathfrak n$ and $\mathfrak n'$ be multisegments with respective segments in the following respective ascending sequences:
\[  \Delta_1', \ldots, \Delta_p', \Delta_1, \ldots, \Delta_r \]
and
\[ \Delta_1'', \ldots, \Delta_q'', \Delta_1, \ldots, \Delta_r .\]
Then $\mathfrak r(\mathfrak n', \pi)=\mathfrak r(\mathfrak n'', \pi)$ if and only if 
\[  \mathfrak r(\left\{ \Delta_1', \ldots, \Delta_p'\right\}, \pi)=\mathfrak r(\left\{ \Delta_1'', \ldots, \Delta_q'' \right\}, \pi) .
\]
\end{proposition}

\begin{proof}
The  if direction is straightforward. We now consider the only if direction. By Theorem \ref{thm realize highest derivative mult}, let $\pi \in \mathrm{Irr}_{\rho}$ such that $\mathfrak{hd}(\pi)=\mathfrak h$. By Theorem \ref{thm isomorphic derivatives}, we have that
\[  D_{\Delta_r}\circ \ldots \circ D_{\Delta_1}\circ D_{\Delta_p'}\circ \ldots \circ D_{\Delta_1'}(\pi) \cong D_{\Delta_r}\circ \ldots \circ D_{\Delta_1}\circ D_{\Delta_q'}\circ \ldots \circ D_{\Delta_1'}(\pi) .
\]
For any irreducible $\tau \in \mathrm{Irr}_{\rho}$ and any segment $\Delta \in \mathrm{Seg}_{\rho}$, denote by $I_{\Delta}(\tau)$ the unique irreducible submodule of $\pi \times \mathrm{St}(\Delta)$. Now, by uniqueness, $I_{\Delta_i}\circ D_{\Delta_i}(\tau)\cong \tau$ if $D_{\Delta_i}(\tau)\neq 0$ for any $i$ and irreducible $\tau$. Hence, we cancel the derivatives $D_{\Delta_r}, \ldots, D_{\Delta_1}$ to obtain:
\[  D_{\Delta_p'}\circ \ldots \circ D_{\Delta_1'}(\pi) \cong D_{\Delta_q'}\circ \ldots \circ D_{\Delta_1'}(\pi) .
\]
\end{proof}

\subsection{First subsequent property}

We shall frequently use the following simple fact:

\begin{lemma} \cite[Section 6.7]{Ze80} \label{lem intersect union ascending order}
Let $\Delta_1, \ldots, \Delta_r$ be a sequence of segments in an ascending order. Suppose $\Delta_k$ and $\Delta_{k+1}$ are linked for some $k$. Then 
\[  \Delta_1, \ldots, \Delta_{k-1}, \Delta_k\cup \Delta_{k+1}, \Delta_k\cap \Delta_{k+1}, \Delta_{k+2}, \ldots, \Delta_r 
\]
is also in an ascending order.
\end{lemma}

A particular case of the lemma is that if $\Delta', \Delta'', \Delta'''$ are in an ascending order, then one has:
\begin{enumerate}
\item $\Delta'\cap \Delta'', \Delta'\cup \Delta'', \Delta'''$ is also in an ascending order;
\item $\Delta', \Delta''\cap \Delta''', \Delta''\cup \Delta'''$ is also in an ascending order.
\end{enumerate}
From these two simple cases, one can deduce some other variations needed in the following Proposition \ref{prop first subseq property}.

\begin{proposition} \label{prop first subseq property}
Let $\mathfrak \pi \in \mathrm{Irr}_{\rho}$. Let $\mathfrak n$ be minimal to $\pi$. We write the segments in $\mathfrak n$ in an ascending order 
\[\Delta_1, \ldots, \Delta_r. \] 
Then, for any $s\leq r$, 
\begin{enumerate}
\item $\left\{\Delta_1, \ldots, \Delta_s \right\}$ is still minimal to $\pi$; and
\item $\left\{ \Delta_{s+1}, \ldots, \Delta_r \right\}$ is minimal to $\mathfrak r(\left\{ \Delta_1, \ldots, \Delta_s\right\}, \pi)$. 
\end{enumerate}
\end{proposition}

\begin{proof}
We only prove (1), and (2) can be proved similarly. 

The admissibility follows from definitions (and Lemma \ref{lem removal process}(5)). Let 
\[  \mathfrak n'=\left\{ \Delta_1, \ldots, \Delta_s \right\} .\]
We pick two linked segments $\Delta_i$ and $\Delta_j$ in $\mathfrak n'$ and we set
\[ \mathfrak n''= \mathfrak n'-\left\{ \Delta_i, \Delta_j \right\}+\Delta_i\cup \Delta_j+\Delta_i\cap \Delta_j  . \]
It suffices to show that $\mathfrak r(\mathfrak n'', \mathfrak h)\neq \mathfrak r(\mathfrak n, \mathfrak h)$. To this end, we first write the segments in $\mathfrak n''$ in an ascending order:
\[  \Delta_1', \ldots, \Delta_s' .
 \]
(There are $s-1$ segments if $\Delta_i \cap \Delta_j =\emptyset$, but the below arguments could be still applied.) It follows from Lemma \ref{lem intersect union ascending order} that $\Delta_1', \ldots, \Delta_s', \Delta_{s+1}, \ldots, \Delta_r$ are still in an ascending order.



Now we return to the proof. The minimality of $\mathfrak n$ implies that
\[  \mathfrak r(\left\{ \Delta_1', \ldots, \Delta_s', \Delta_{s+1}, \ldots, \Delta_r\right\}, \pi) \neq \mathfrak r(\left\{ \Delta_1, \ldots, \Delta_s, \Delta_{s+1}, \ldots, \Delta_r\right\}, \pi).
\]
By Proposition \ref{prop cancel property}, 
\[  \mathfrak r(\left\{ \Delta_1', \ldots, \Delta_s'\right\}, \pi) \neq \mathfrak r(\left\{ \Delta_1, \ldots, \Delta_s \right\}, \pi),
\]
as desired.
\end{proof}

\section{Three segment basic cases} \label{s three segment cases}

The main goal of this section is to prove the subsequent property and commutativity for three segment cases. To show the minimality, the main strategy is to use the convex property for $\mathcal S(\pi, \tau)$ of Theorem \ref{thm convex derivatives} and the non-overlapping property of Theorem \ref{prop nonoverlapping property}.

\subsection{Case: $\left\{ \Delta_1, \Delta_3\right\}$ minimal to $D_{\Delta_2}(\pi)$}

\begin{lemma} \label{lem eta change unlinked case}
Let $\Delta, \Delta' \in \mathrm{Seg}_{\rho}$ with $\Delta \subset \Delta'$. Let $\pi \in \mathrm{Irr}_{\rho}$ with $D_{\Delta'}(\pi)\neq 0$. Then, the following holds: 
\begin{enumerate}
\item If $a(\Delta)>a(\Delta')$, then $\eta_{\Delta}(D_{\Delta'}(\pi))=\eta_{\Delta}(\pi)$.
\item If $a(\Delta) = a(\Delta')$, then $\eta_{\Delta}(D_{\Delta'}(\pi))$ is obtained from $\eta_{\Delta}(\pi)$ by decreasing the coordinate $\varepsilon_{\Delta}(\pi)$ by $1$. 
\end{enumerate}
\end{lemma}

\begin{proof}
By Theorem \ref{thm effect of Steinberg}, it suffices to compare $\eta_{\Delta}(\mathfrak r(\Delta', \pi))$ and $\eta_{\Delta}(\pi)$. Let the removal sequence for $(\Delta', \pi)$ be 
\[  \Delta_1, \ldots, \Delta_r .
\]
For (1), we consider two cases.
\begin{itemize}
\item[(i)] Suppose there does not exist an integer $i^*$ such that $a(\Delta_{i^*}) \geq a(\Delta)$ and $a(\Delta_{i^*})\leq b(\Delta)$. In such case, all $\Delta_1, \ldots, \Delta_r$ and $\Delta_1^{tr}, \ldots, \Delta_r^{tr}$ do not contribute $\eta_{\Delta}$. Thus we have such equality. 
\item[(ii)] Suppose there exists an integer $i^*$ such that $a(\Delta_{i^*})\geq a(\Delta)$. Let $i^*>1$ be the smallest such integer. Let $j^*$ be the largest integer such that $a(\Delta_{j^*}) \leq b(\Delta)$. We have that $\Delta_{i^*}, \ldots, \Delta_{j^*}$ are all the segments in the removal sequence contributing to $\eta_{\Delta}(\pi)$ and $\Delta^{tr}_{i^*-1}, \ldots, \Delta^{tr}_{j^*-1}$ are all the segments in the truncated one contributing to $\eta_{\Delta}(\mathfrak r(\Delta', \pi))$. Note that, for $i^*\leq k \leq j^*$, $\Delta_k$ and $\Delta_{k-1}^{tr}$ contribute to the same coordinate $\varepsilon_{\widetilde{\Delta}}$ for some segment $\widetilde{\Delta}$. This shows the equality to two $\eta_{\Delta}$.
\end{itemize}

For (2), it is similar, but $i^*$ in above notation becomes $1$. Again, for $2\leq k \leq j^*$, $\Delta_k$ and $\Delta_{k-1}^{tr}$ contribute to the same $\varepsilon$. The term $\Delta_1$ explains $\varepsilon_{\Delta}(\pi)$ is decreased by $1$ to obtain $\varepsilon_{\Delta}(D_{\Delta'}(\pi))$.  
\end{proof}

\begin{lemma} \label{lem minimal in a basic case}
Let $\mathfrak m=\left\{ \Delta_1, \Delta_2, \Delta_3 \right\} \in \mathrm{Mult}_{\rho}$ in an ascending order. Let $\pi \in \mathrm{Irr}_{\rho}$ be such that $\mathfrak m$ is minimal to $\pi$. Then $\left\{ \Delta_1, \Delta_3\right\}$ is also minimal to $D_{\Delta_2}(\pi)$.
\end{lemma}

\begin{proof}
By Proposition \ref{prop first subseq property}, $\left\{\Delta_1, \Delta_2 \right\}$ is minimal to $\pi$. By Proposition \ref{prop dagger property 2} (for the linked case between $\Delta_1$ and $\Delta_2$) and Lemma \ref{lem comm derivative 1} (for the unlinked case between $\Delta_1$ and $\Delta_2$), we have that
\[  D_{\Delta_3}\circ D_{\Delta_2}\circ D_{\Delta_1}(\pi) \cong D_{\Delta_3}\circ D_{\Delta_1}\circ D_{\Delta_2}(\pi) .
\]
The minimality is automatic if $\Delta_1$ and $\Delta_3$ are unlinked. So we shall assume that $\Delta_1$ and $\Delta_3$ are linked. There are three possibilities:
\begin{itemize}
\item $\Delta_2$ is unlinked to $\Delta_1$ (and so $\left\{ \Delta_2, \Delta_1, \Delta_3\right\}$ is still in an ascending order). Then the minimality of $\mathfrak m$ and Proposition \ref{prop first subseq property} imply this case.
\item $\Delta_2$ is unlinked to $\Delta_3$. We consider following possibilities:
\begin{itemize}
  \item[(i)] $\Delta_3 \subset \Delta_2$. Then $\left\{ \Delta_1, \Delta_3 \right\}$ is also minimal to $\pi$ by Lemma \ref{lem comm derivative 1} and Proposition \ref{prop first subseq property}. Hence, $\eta_{\Delta_3}(D_{\Delta_1}(\pi))=\eta_{\Delta_3}(\pi)$. On the other hand, if $a(\Delta_3)> a(\Delta_2)$, by Lemma \ref{lem eta change unlinked case}(1), 
	\[ \eta_{\Delta_3}(D_{\Delta_2}(\pi))=\eta_{\Delta_3}(\pi) \]
		and 
	\[ \eta_{\Delta_3}(D_{\Delta_2}\circ D_{\Delta_1}(\pi))=\eta_{\Delta_3}(D_{\Delta_1}(\pi)). \] Combining two equations, we have 
	\[ \eta_{\Delta_3}(D_{\Delta_2} \circ D_{\Delta_1}(\pi)) =\eta_{\Delta_3}(D_{\Delta_2}(\pi)) 
	\]
	and so, by Proposition \ref{prop dagger property 2} and Lemma \ref{lem comm derivative 1} again,
	\[  \eta_{\Delta_3}(D_{\Delta_1}\circ D_{\Delta_2}(\pi)) =\eta_{\Delta_3}(D_{\Delta_2}(\pi)).
	\]
	Thus, we have the minimality by Proposition \ref{prop nonoverlapping property}.
	
	When $a(\Delta_3)\cong a(\Delta_2)$, the argument is similar. The only difference, by Lemma \ref{lem eta change unlinked case}(2), is that $\eta_{\Delta_3}(D_{\Delta_2}(\pi))$ (resp. $\eta_{\Delta_3}(D_{\Delta_2}\circ D_{\Delta_1}(\pi))$) is obtained from $\eta_{\Delta_3}(\pi)$ (resp. $\eta_{\Delta_3}(D_{\Delta_1}(\pi))$) by decreasing the $\varepsilon_{\Delta_3}(\pi)$ (resp. $\varepsilon_{\Delta_3}(D_{\Delta_1}(\pi))$) by $1$.
	\item[(ii)] $\Delta_2 \subset \Delta_3$. In such case, one first has that $(\Delta_1, \Delta_3, \mathfrak{hd}(\pi))$ satisfies the non-overlapping property by the minimality of $\left\{ \Delta_1, \Delta_3\right\}$ to $\pi$. Then, by Lemma \ref{lem dagger property on derivative} to show that $(\Delta_1, \Delta_3, \mathfrak{hd}(D_{\Delta_2}(\pi)))$ still satisfies the non-overlapping property. Thus, $\left\{ \Delta_1, \Delta_3 \right\}$ is minimal to $D_{\Delta_2}(\pi)$. 
	\item[(iii)] $b(\Delta_2)<a(\Delta_3)$. Then the ascending order and linkedness between $\Delta_1$ and $\Delta_3$ also give that $\Delta_1$ and $\Delta_2$ are not linked. This goes back to the above bullet.
	\item[(iv)] $b(\Delta_3)<a(\Delta_2)$. Then the ascending order and linkedness between $\Delta_1$ and $\Delta_3$ also give that $\Delta_1$ and $\Delta_2$ are not linked. This goes back to the above bullet.
\end{itemize}
\item $\Delta_1<\Delta_2<\Delta_3$. If $\left\{ \Delta_1, \Delta_3\right\}$ is not minimal to $D_{\Delta_2}(\pi)$, then 
\[   D_{\Delta_1\cup \Delta_3}\circ D_{\Delta_1\cap \Delta_3}\circ D_{\Delta_2}(\pi)\cong D_{\Delta_3}\circ D_{\Delta_1}\circ D_{\Delta_2}(\pi)\cong D_{\Delta_3}\circ D_{\Delta_2}\circ D_{\Delta_1}(\pi) . \]
This contradicts the minimality of $\mathfrak m$.
\end{itemize}
\end{proof}

\subsection{Case: $\left\{ \Delta_2, \Delta_3\right\}$ minimal to $\pi$}

\begin{lemma} \label{lem basic subsequent property}
Let $\pi \in \mathrm{Irr}_{\rho}$. Let $\Delta_1, \Delta_2, \Delta_3$ be segments in an ascending order. If $\left\{ \Delta_1, \Delta_2, \Delta_3 \right\}$ is minimal to $\pi$, then $\left\{ \Delta_2, \Delta_3 \right\}$ is also minimal to $\pi$. 
\end{lemma}

\begin{proof}
When $\Delta_2$ and $\Delta_3$ are unlinked, there is nothing to prove. We assume that $\Delta_2$ and $\Delta_3$ are linked and so $\Delta_2<\Delta_3$. We consider the following cases:
\begin{itemize}
\item $b(\Delta_3) \leq b(\Delta_1)$. Then $\Delta_1$ is unlinked to both $\Delta_2$ and $\Delta_3$. Then the minimality of $\left\{ \Delta_2, \Delta_3 \right\}$ to $\pi$ follows from the minimality of $\left\{ \Delta_1, \Delta_2, \Delta_3 \right\}$ to $\pi$ by Proposition \ref{prop first subseq property}. 
\item $b(\Delta_1)<b(\Delta_3)$ and $a(\Delta_1)<a(\Delta_3)$, and $\Delta_1$ and $\Delta_3$ are linked. Then $\left\{ \Delta_1, \Delta_2\right\}$ with $\Delta_3$ is in the situation of Proposition \ref{prop dagger property}. Then $\eta_{\Delta_3}(D_{\left\{ \Delta_1, \Delta_2 \right\}}(\pi))=\eta_{\Delta_3}(\pi)$. On the other hand, we have $\eta_{\Delta_3}(D_{\Delta_1}\circ D_{\Delta_2}(\pi))=\eta_{\Delta_3}(D_{\Delta_2}(\pi))$ by Lemma \ref{lem minimal in a basic case} with Proposition \ref{prop nonoverlapping property}(1)$\Leftrightarrow$(3). Furthermore, $D_{\Delta_1}\circ D_{\Delta_2}(\pi)\cong D_{\left\{\Delta_1, \Delta_2 \right\}}(\pi)$, which again follows from Proposition \ref{prop dagger property 2}  and Lemma \ref{lem comm derivative 1} (see the top of the proof of Lemma \ref{lem minimal in a basic case} for a bit more details). 

Thus, combining above, we have $\eta_{\Delta_3}(\pi)=\eta_{\Delta_3}(D_{\Delta_2}(\pi))$. This implies $\left\{ \Delta_2, \Delta_3\right\}$ is minimal to $\pi$ by Proposition \ref{prop nonoverlapping property}(1)$\Leftrightarrow$(3).
\item $b(\Delta_1)<b(\Delta_3)$ and $a(\Delta_1)<a(\Delta_3)$, and $\Delta_1$ and $\Delta_3$ are not linked. Indeed, the argument in above bullet still works except that one has to directly observe from the removal process that  $\eta_{\Delta_3}(D_{\Delta_1}\circ D_{\Delta_2}(\pi))=\eta_{\Delta_3}(D_{\Delta_2}(\pi))$.
\item $b(\Delta_1)<b(\Delta_3)$ and $a(\Delta_3)\leq a(\Delta_1)$. In particular, $\Delta_1$ is unlinked to $\Delta_3$. Since we are assuming $\Delta_1, \Delta_2, \Delta_3$ are in an ascending order, and we are assuming that $\Delta_2$ and $\Delta_3$ are linked, we also have that $\Delta_1$ and $\Delta_2$ are unlinked. Then the minimality of $\left\{ \Delta_2, \Delta_3 \right\}$ to $\pi$ follows from the minimality of $\left\{ \Delta_1, \Delta_2, \Delta_3 \right\}$ to $\pi$ and Proposition \ref{prop first subseq property}. 
\end{itemize}
\end{proof}

\subsection{Case: $\left\{ \Delta_1, \Delta_3\right\}$ minimal to $\pi$} \label{ss case 13 minimal to pi}

\begin{lemma} \label{lem basic subsequent property 3}
Let $\pi \in \mathrm{Irr}_{\rho}$. Let $\Delta_1, \Delta_2, \Delta_3$ be segments in an ascending order. If $\left\{ \Delta_1, \Delta_2, \Delta_3 \right\}$ is minimal to $\pi$, then $\left\{ \Delta_1, \Delta_3\right\}$ is minimal to $\pi$.
\end{lemma}

\begin{proof}
We may assume $\Delta_1$ and $\Delta_3$ are linked. Otherwise, there is nothing to prove. We consider the following cases.
\begin{itemize}
\item   $\Delta_1<\Delta_2<\Delta_3$. By Proposition \ref{prop dagger property 2}, we have
\[    D_{\left\{ \Delta_1, \Delta_3, \Delta_2 \right\}}(\pi)\cong D_{\Delta_2}\circ D_{\left\{ \Delta_1, \Delta_3 \right\}}(\pi) .
\]
If the minimality does not hold, then we have
\[D_{\left\{ \Delta_1, \Delta_2, \Delta_3 \right\}}(\pi) \cong D_{\left\{\Delta_2, \Delta_1\cup \Delta_3, \Delta_1\cap \Delta_3 \right\}}(\pi) 
\]
since $\Delta_1\cup \Delta_3, \Delta_1\cap \Delta_3, \Delta_2$ are still in an ascending order. This contradicts to the minimality of $\left\{ \Delta_1, \Delta_2, \Delta_3\right\}$ to $\pi$.
\item $\Delta_2$ and $\Delta_3$ are not linked. Then we can switch the labellings of $\Delta_2$ and $\Delta_3$, which gives the minimality of $\left\{ \Delta_1, \Delta_3 \right\}$ to $\pi$ by Proposition \ref{prop first subseq property}.
\item  $\Delta_1$ and $\Delta_2$ are not linked. In this case, we can switch the labellings for $\Delta_1$ and $\Delta_2$ by using linkedness. Then the result follows from Lemma \ref{lem basic subsequent property}.
\end{itemize}
\end{proof}

\subsection{Case: $\left\{ \Delta_1, \Delta_2 \right\}$ minimal to $D_{\Delta_3}(\pi)$} \label{ss basis case 1,3 to D3}


We now need some inputs from representation theory to prove a combinatorics result. Let $N_i \subset G_n$ (depending on $n$) be the unipotent radical containing matrices of the form $\begin{pmatrix} I_{n-i} & * \\ & I_i \end
{pmatrix}$. For a smooth representation $\pi$ of $G_n$, we write $\pi_{N_i}$ to be its Jacquet module.

\begin{lemma} \label{lem irreducible of a product}
Let $\Delta=[a,b]_{\rho}, \Delta'=[a',b']_{\rho}$ be two segments such that $\Delta < \Delta'$. Let $\omega= \mathrm{St}(\left\{ \Delta, \Delta' \right\})$. Let $\mathfrak m$ be a multisegment whose segments $\Delta''=[a'',b'']_{\rho}$ satisfy that $b''=b$ and $a'' < a$. Then $\mathrm{St}(\mathfrak m) \times \omega$ is irreducible and 
\[  \mathrm{St}(\mathfrak m)\times \omega \cong \omega \times \mathrm{St}(\mathfrak m) . \]
\end{lemma}

\begin{proof}
It is well-known that the second assertion implies the first one. We only have to prove the first one. Since $\mathrm{St}(\mathfrak m)$ can be written as $\times_{\Delta \in \mathfrak m} \mathrm{St}(\Delta)$, it reduces to the case that $\mathfrak m$ contains only one segment and so now we consider $\mathfrak m=\left\{ \widetilde{\Delta} \right\}$. 

We analyse possible composition factors of 
\[   \mathrm{St}(\widetilde{\Delta}) \times \omega .
\]
Since we know that a composition factor of $\mathrm{St}(\widetilde{\Delta})\times \omega$ is also a composition factor of $\mathrm{St}(\widetilde{\Delta})\times \mathrm{St}(\Delta)\times \mathrm{St}(\Delta')$, the possible composition factors are 
\[  \mathrm{St}(\left\{ \widetilde{\Delta}, \Delta, \Delta' \right\}), \mathrm{St}(\widetilde{\Delta}\cup \Delta'+\widetilde{\Delta}\cap \Delta'+\Delta), \mathrm{St}(\Delta \cup \Delta'+\Delta\cap \Delta'+\widetilde{\Delta}) .
\]
We denote the three representations $\pi_1, \pi_2, \pi_3$ respectively.

Thus it suffices to show that the last two composition factors cannot appear in $\mathrm{St}(\widetilde{\Delta}) \times \omega$. We first consider $\pi_2$. Note that $\pi_2$ is generic. However, $\omega$ is not generic and so $\mathrm{St}(\widetilde{\Delta}) \times \omega$ cannot contains a generic composition factor and so $\pi_2$ cannot appear in $\mathrm{St}(\widetilde{\Delta})\times \omega$. 

We now consider $\pi_3$. Let $l=l_{abs}(\Delta \cup \Delta')$. Then $(\pi_3)_{N_l}$ has the composition factor $\mathrm{St}(\Delta ) \boxtimes \mathrm{St}(\Delta \cup \Delta')$. Now we consider composition factors in $(\mathrm{St}(\widetilde{\Delta})\times \omega)_{N_l}$. If $(\mathrm{St}(\widetilde{\Delta})\times \omega)_{N_l}$ contains the factor $\mathrm{St}(\Delta ) \boxtimes \mathrm{St}(\Delta \cup \Delta')$, a simple composition factor is a simple composition factor in 
\[   \mathrm{St}(\Delta) \times \omega_1 \boxtimes \omega_2 ,
\]
where $\omega_1 \boxtimes \omega_2$ is a simple composition factor in $\omega_{N_l}$. However, the possibilities of those composition factors are well-known and it is impossible for $\omega_2$ to be the factor $\mathrm{St}(\Delta \cup \Delta')$. 
\end{proof}

\begin{lemma} \label{lem intermediate segment property for delta3}
Let $\pi \in \mathrm{Irr}_{\rho}$. Let $\Delta_1, \Delta_2, \Delta_3$ be segments satisfying $\Delta_1<\Delta_2<\Delta_3$. Suppose $\left\{ \Delta_1, \Delta_2, \Delta_3 \right\}$ is minimal to $\pi$. Then $D_{\Delta_3}(\pi)$. Let $\widetilde{\Delta}=\Delta_1 \cup \Delta_2$. Then the followings hold:
\begin{itemize}
\item[(1)] $\eta_{\widetilde{\Delta}}(\pi) -\eta_{\Delta_2}(\pi) = \eta_{\widetilde{\Delta}}(D_{\Delta_3}(\pi))-\eta_{\Delta_2}(D_{\Delta_3}(\pi))$;
\item[(2)] If $(\Delta_1, \Delta_2, \pi)$ satisfies the intermediate segment property, then $(\Delta_1, \Delta_2, D_{\Delta_3}(\pi))$ also satisfies the intermediate segment property. 
\end{itemize}
Here the subtraction in (1) means the subtraction entry-wise.
\end{lemma}

\begin{proof}
We have shown in Lemma \ref{lem basic subsequent property} that $\left\{ \Delta_2, \Delta_3 \right\}$ is minimal to $\pi$. Thus, we have
\[  D_{\Delta_3}\circ D_{\Delta_2}(\pi) \not\cong D_{\Delta_2\cup \Delta_3}\circ D_{\Delta_2\cap \Delta_3}(\pi) .\]
By a standard argument (see e.g. proof of \cite[Proposition 6.2]{Ch22+d}), we have that $\pi$ is the unique simple submodule of 
\[   D_{\Delta_3}\circ D_{\Delta_2}(\pi) \times \mathrm{St}(\left\{ \Delta_2, \Delta_3 \right\}) .\]

Write $\Delta_2=[a_2,b_2]_{\rho}$. Let 
\[   \mathfrak m = \sum_{c<a_2} \varepsilon_{[c, b_2]_{\rho}}(D_{\Delta_3}(\pi)) \cdot [c,b_2]_{\rho} , 
\]
and 
\[  \mathfrak p = \sum_{c< a_2} \varepsilon_{[c,b_2]_{\rho}}(\pi) \cdot [c,b_2]_{\rho} .
\]

Thus, we have:
\[   D_{\Delta_3}\circ D_{\Delta_2}(\pi) \hookrightarrow D_{\mathfrak m}\circ D_{\Delta_3}\circ D_{\Delta_2}(\pi)\times \mathrm{St}(\mathfrak m)
\]
and so 
\[ \pi \hookrightarrow D_{\Delta_2}\circ D_{\Delta_2}(\pi) \times \mathrm{St}(\left\{ \Delta_2, \Delta_3 \right\}) \hookrightarrow D_{\mathfrak m}\circ D_{\Delta_3}\circ D_{\Delta_2}(\pi) \times \mathrm{St}(\mathfrak m) \times \mathrm{St}(\left\{ \Delta_2, \Delta_3 \right\}).
\]
By Lemma \ref{lem irreducible of a product}, 
\[  \pi \hookrightarrow D_{\mathfrak m}\circ D_{\Delta_3}\circ D_{\Delta_2}(\pi)  \times \mathrm{St}(\left\{ \Delta_2, \Delta_3 \right\}) \times \mathrm{St}(\mathfrak m).
\]
This implies that $D_{\mathfrak m}(\pi)\neq 0$ and so $\mathfrak m$ is a submultisegment of $\mathfrak p$. On the other hand, using Lemma \ref{lem removal process} and Theorem \ref{thm effect of Steinberg small a}, we also have that $\mathfrak p$ is a submultisegment of $\mathfrak m$. Hence, $\mathfrak m=\mathfrak p$. Translating to $\eta$-invariants, we obtain (1). 

We now prove (2). Write $\Delta_1=[a_1,b_1]_{\rho}$. Suppose $(\Delta_1, \Delta_2, \pi)$ satisfies the intermediate segment property. Then there exists a segment $[a,b]_{\rho}$ in $\mathfrak{hd}(\pi)$ satisfying that $a_1 \leq a < a_2$ and $b_1 \leq b <b_2$. Then, 
\[  \varepsilon_{[a,b]_{\rho}}(\pi) > \varepsilon_{[a,b_2]_{\rho}}(\pi) \]
By (1), we have that 
\[ \varepsilon_{[a,b_2]_{\rho}}(\pi)=\varepsilon_{[a,b_2]_{\rho}}(D_{\Delta_3}(\pi)).\]
By Theorem \ref{thm effect of Steinberg small a}, 
\[  \varepsilon_{[a,b]_{\rho}}(D_{\Delta_3}(\pi)) \geq \varepsilon_{[a,b]_{\rho}}(\pi) .\]
Combining the above equalities and inequalities, we have:
\[  \varepsilon_{[a,b]_{\rho}}(D_{\Delta_3}(\pi)) > \varepsilon_{[a,b_2]_{\rho}}(D_{\Delta}(\pi)) . \]
This implies that there exists a segment $[a,b']_{\rho}$ in $\mathfrak{hd}(D_{\Delta_3}(\pi))$ with $b'<b_2$. Thus $(\Delta_1, \Delta_2, D_{\Delta_3}(\pi))$ satisfies the intermediate segment property.
\end{proof}

It is of course desirable to also have a more combinatorial proof for Lemme \ref{lem intermediate segment property for delta3} while it seems to require some further developments on removal process to do so.

\begin{lemma} \label{lem basic subsequent property 2}
Let $\pi \in \mathrm{Irr}_{\rho}$. Let $\Delta_1, \Delta_2, \Delta_3$ be segments in an ascending order. If $\left\{ \Delta_1, \Delta_2, \Delta_3 \right\}$ is minimal to $\pi$, then $\left\{ \Delta_1, \Delta_2 \right\}$ is also minimal to $D_{\Delta_3}(\pi)$. 
\end{lemma}

\begin{proof}
If $\Delta_1$ and $\Delta_2$ are unlinked, then there is nothing to prove. If $\Delta_2$ and $\Delta_3$ are unlinked, then we use Lemma \ref{lem comm derivative 1} to transfer to Lemma \ref{lem minimal in a basic case}.

The remaining case is that $\Delta_1$ and $\Delta_2$ are linked, and $\Delta_2$ and $\Delta_3$ are linked. In other words, $\Delta_1<\Delta_2<\Delta_3$. This case follows from Proposition \ref{prop nonoverlapping property}(2)$\Leftrightarrow$(3) and Lemma \ref{lem intermediate segment property for delta3}.
\end{proof}

\subsection{Case: $\left\{ \Delta_1, \Delta_2 \right\}$ minimal to $\pi$}

\begin{lemma} \label{lem basic subsequent property 12 }
Let $\pi \in \mathrm{Irr}_{\rho}$. Let $\Delta_1, \Delta_2, \Delta_3$ be segments in an ascending order. If $\left\{ \Delta_1, \Delta_2, \Delta_3 \right\}$ is minimal to $\pi$, then $\left\{ \Delta_1, \Delta_2 \right\}$ is also minimal to $\pi$. 
\end{lemma}

The above lemma is a special case of Proposition \ref{prop first subseq property}(1).

\subsection{Case: $\left\{ \Delta_2, \Delta_3 \right\}$ minimal to $D_{\Delta_1}(\pi)$}

\begin{lemma} \label{lem basic subsequent property 23 1}
Let $\pi \in \mathrm{Irr}_{\rho}$. Let $\Delta_1, \Delta_2, \Delta_3$ be segments in an ascending order. If $\left\{ \Delta_1, \Delta_2, \Delta_3 \right\}$ is minimal to $\pi$, then $\left\{ \Delta_2, \Delta_3 \right\}$ is also minimal to $D_{\Delta_1}(\pi)$. 
\end{lemma}

The above lemma is again a special case of Proposition \ref{prop first subseq property}(2).



\section{Subsequent property of minimal sequences} \label{s subseq property}

\subsection{Consecutive pairs}

\begin{definition} \label{def consecutive pair}
Let $\mathfrak m \in \mathrm{Mult}_{\rho}$. Two segments $\Delta_1$ and $\Delta_2$ in $\mathfrak m$ are said to be {\it consecutive} in $\mathfrak m$ if 
\begin{itemize}
\item $\Delta_1<\Delta_2$ i.e. $\Delta_1$ and $\Delta_2$ are linked with $a(\Delta_1)<a(\Delta_2)$
\item there is no other segment $\Delta'$ in $\mathfrak m$ such that 
\[ a(\Delta_1)\leq a(\Delta') \leq a(\Delta_2), \quad b(\Delta_1) \leq b(\Delta') \leq b(\Delta_2) \]
and $\Delta'$ is linked to either $\Delta_1$ or $\Delta_2$. 
\end{itemize}
(The last linkedness condition guarantees that $\Delta'\neq \Delta_1\cap \Delta_2$ and $\Delta'\neq \Delta_1\cup \Delta_2$.)
\end{definition}


\begin{example}
\begin{itemize}
\item Let $\mathfrak h=\left\{ [0,3], [1,4], [2,5] \right\}$. Then $[0,3], [1,4]$ form a pair of consecutive segments. Similarly, $[1,4], [2,5]$ also form a pair of consecutive segments, but $[0,3], [2,5]$ do not form a pair of consecutive segments.
\item Let $\mathfrak h=\left\{ [0,4], [1,2], [2,5] \right\}$. Then $[0,4], [2,5]$ form a pair of consecutive segments; and $[1,2], [2,5]$ also form a pair of consecutive segments.
\item Let $\mathfrak h=\left\{ [0,3],[1,3], [2,4], [2,5] \right\}$. Then $[1,3], [2,4]$ form a pair of consecutive segments, while $[0,3], [2,4]$ do not form a pair of consecutive segments.
\end{itemize}
\end{example}

The terminology of consecutive segments is suggested by its property in the intersection-union process.

\begin{lemma}
Let $\Delta_1, \Delta_2$ be linked segments with $\Delta_1 < \Delta_2$. Suppose there exists a segment $\Delta'$ satisfying the conditions in the second bullet of Definition \ref{def consecutive pair}. Then, if $\Delta'$ is linked to $\Delta_i$ ($i=1,2$), then 
\[   \left\{ \Delta_1 \cap \Delta_2, \Delta_1\cup \Delta_2, \Delta' \right\} \leq_Z \left\{ \Delta_i \cap \Delta', \Delta_i\cup \Delta', \Delta_{j}  \right\} \leq_Z \left\{ \Delta_1, \Delta_2, \Delta' \right\}.
\]
Here $j$ is the index other than $i$ i.e. $j\in \left\{ 1,2\right\}-\left\{i \right\}$. 
\end{lemma}

The above lemma follows from a direct checking and we omit the details. A simple combinatorics give the following:

\begin{lemma} \label{lem cover consecutive}
\begin{enumerate}
\item Let $\mathfrak m, \mathfrak m' \in \mathrm{Mult}_{\rho}$ such that $\mathfrak m' \leq_Z \mathfrak m$ with $\mathfrak m'\neq \mathfrak m$. Then there exists a pair of consecutive segments $\Delta, \Delta'$ in $\mathfrak m$ such that for the multisegment $\mathfrak m''$ obtained from $\mathfrak m$ by the elementary intersection-union process involving $\Delta$ and $\Delta'$, $\mathfrak m' \leq_Z \mathfrak m'' \leq_Z \mathfrak m$.
\item Let $\mathfrak m \in \mathrm{Mult}_{\rho}$. Let $\Delta, \Delta'$ be a pair of consecutive segments in $\mathfrak m$ with $\Delta <\Delta'$. Let $\mathfrak m'$ be the submultisegment of $\mathfrak m-\Delta-\Delta'$ that contains all the segments $\widetilde{\Delta}$ with $a(\Delta')\leq a(\widetilde{\Delta})$ or $b(\Delta') \leq b(\widetilde{\Delta})$. Write the segments in $\mathfrak m-\mathfrak m'-\Delta-\Delta'$ in an ascending order: $\Delta_1, \ldots, \Delta_r$ and write the segments in $\mathfrak m'$ in an ascending order: $\Delta_1', \ldots, \Delta_s'$. Then the sequence:
\[  \Delta_1, \ldots, \Delta_r, \Delta, \Delta', \Delta_1', \ldots, \Delta_s' \]
is ascending.
\end{enumerate}
\end{lemma}

\begin{proof}
For (1), it suffices to show for $\mathfrak m'$ obtained by a pair of elementary intersection-union operation involving $\overline{\Delta}$ and $\overline{\Delta}'$. If the segments involved in the operation are consecutive, then the statement is immediate. Otherwise, there exists a segment $\widetilde{\Delta}$ such that $\widetilde{\Delta}$ is linked to either $\overline{\Delta}$ or $\overline{\Delta}'$, and produce a multisegment $\widetilde{\mathfrak m}$ such that $\mathfrak m'\leq_Z \widetilde{\mathfrak m}\leq_Z \mathfrak m''$. We repeat the process if such linked pair is still not consecutive. Note that if $\widetilde{\Delta}$ is linked to $\overline{\Delta}$ (resp. $\overline{\Delta}'$), we must have $\widetilde{\Delta} \cap \overline{\Delta}$ (resp. $\widetilde{\Delta} \cap \overline{\Delta}'$) strictly longer than $\overline{\Delta} \cap \overline{\Delta}'$, and hence after repeating the process several times, we obtain desired consecutive segments.

For (2), it is a direct check from the definition of an ascending order.
\end{proof}

\subsection{Minimality under commutativity (second basic case)}

We first prove a commutativity result of minimal sequences, which is useful in proving Theorem \ref{thm subsequent minimal}.

\begin{lemma} \label{lem min commut in special case}
Let $\pi \in \mathrm{Irr}$. Let $\Delta, \Delta_1, \Delta_2, \ldots, \Delta_r$ be in an ascending order and minimal to $\pi$. Then $\left\{ \Delta, \Delta_{k+1}, \ldots, \Delta_r\right\}$ is also minimal to $D_{\Delta_k}\circ \ldots \circ D_{\Delta_1}(\pi)\neq 0$, and
\[   D_{\Delta_r}\circ \ldots \circ D_{\Delta_{k+1}}\circ D_{\Delta}\circ D_{\Delta_k}\circ \ldots D_{\Delta_1}(\pi)\cong D_{\Delta_r}\circ \ldots \circ D_{\Delta_1}\circ D_{\Delta}(\pi) .
\]
\end{lemma}

\begin{proof}
By induction, it suffices to show when $k=1$. The case that $\Delta$ and $\Delta_1$ are unlinked is easy by Lemma \ref{lem comm derivative 1} and Proposition \ref{prop first subseq property}(1). Suppose $\left\{ \Delta, \Delta_2, \ldots, \Delta_r\right\}$ is not minimal to $D_{\Delta_1}(\pi)$ to arrive a contradiction. By Theorem \ref{thm convex derivatives} and Lemma \ref{lem cover consecutive}(1), there exists a pair of consecutive segments $\Delta'<\Delta''$ such that the multisegment obtained from the intersection-union operation of those two segments $\Delta', \Delta''$ gives the same derivative on $\pi$. 

Note that $\left\{ \Delta_2, \ldots, \Delta_r \right\}$ is minimal to $D_{\Delta_1}\circ D_{\Delta}(\pi) \cong D_{\Delta}\circ D_{\Delta_1}(\pi)$ (the last isomorphism by Proposition \ref{prop dagger property 2}). Thus the remaining possible cases to be considered could be that one of $\Delta', \Delta''$ is $\Delta$ and so $\Delta'=\Delta$. Let 
\[  \mathfrak n=\left\{ \Delta_i :  a(\Delta_i)\leq a(\Delta) \quad \mbox{or} \quad b(\Delta_i) \leq b(\Delta) \right\}
\]
Then any $\Delta_i$ is unlinked to $\Delta$ and $\Delta_1$ by using the ascending property for the sequence $\Delta, \Delta_1, \ldots, \Delta_r$ (and $\Delta< \Delta_1$). (In particular, $\Delta_1$ is not in $\mathfrak n$.)

 Let $\mathfrak m=\left\{ \Delta_2, \ldots, \Delta_r\right\}$. Now, since we chose $\Delta$ and $\Delta''$ to be consecutive, we can arrange and relabel the segments in an ascending order:
\[                   \widetilde{\Delta}_1, \ldots, \widetilde{\Delta}_l, \Delta, \Delta'', \widetilde{\Delta}_{k+1}, \ldots, \widetilde{\Delta}_{r-1},
\]
where all $\widetilde{\Delta}_1, \ldots, \widetilde{\Delta}_l$ are all elements in $\mathfrak n$ and $\widetilde{\Delta}_{k+1}, \ldots, \widetilde{\Delta}_{r-1}$ are all elements in $\mathfrak m-\mathfrak n$.
 
\noindent
{\bf Case 1:} $\mathfrak n=\emptyset$. We still have that $ \Delta, \Delta_1, \Delta'', \widetilde{\Delta}_1, \ldots, \widetilde{\Delta}_{r-1} $ form an ascending order and is minimal to $\pi$. In particular, we have $\Delta, \Delta_1, \Delta''$ is minimal to $\pi$ by the cancellative property. Thus
\[    D_{\Delta''}\circ D_{\Delta}\circ D_{\Delta_1}(\pi) \not\cong D_{\Delta''\cup \Delta}\circ D_{\Delta''\cap \Delta}\circ D_{\Delta_1}(\pi) 
\]
by Lemma \ref{lem minimal in a basic case} and Proposition \ref{prop dagger property 2}. Since  
\[  \Delta''\cap \Delta,  \Delta''\cup \Delta, \widetilde{\Delta}_{1}, \ldots, \widetilde{\Delta}_{r-1}
\]
still form an ascending order (Lemma \ref{lem intersect union ascending order}), applying $D_{\widetilde{\Delta}_1}, \ldots, D_{\widetilde{\Delta}_{r-1}}$ gives different derivatives on $\pi$ (Proposition \ref{prop cancel property}) and so this gives a contradiction.  \\

\noindent
{\bf Case 2:} $\mathfrak n\neq \emptyset$.
This implies that $\Delta+\mathfrak m-\mathfrak n$ is not minimal to $D_{\mathfrak n}\circ D_{\Delta_1}(\pi)$. However, we have that, by using unlinkedness discussed in the second paragraph, 
\[  D_{\mathfrak n}\circ D_{\Delta_1}(\pi) \cong D_{\Delta_1} \circ D_{\mathfrak n}(\pi) .
\]
and $\Delta+\Delta_1+\mathfrak m-\mathfrak n$ is minimal to $D_{\mathfrak n}(\pi)$ by Proposition \ref{prop first subseq property}. However, from Case 1, we have that $\Delta+\mathfrak m-\mathfrak n$ is minimal to 
\[   D_{\Delta_1}\circ D_{\mathfrak n}(\pi) (\cong D_{\mathfrak n}\circ D_{\Delta_1}(\pi)).
\]
 This gives a contradiction. 
\end{proof}

\subsection{Minimality of a subsequent sequence}

\begin{theorem} \label{thm subsequent minimal}
Let $\pi \in \mathrm{Irr}_{\rho}$. Let $\mathfrak n \in \mathrm {Mult}_{\rho}$ be minimal to $\pi$. Then any submultisegment of $\mathfrak n$ is also minimal to $\pi$.
\end{theorem}

\begin{proof}

By an induction, it suffices to show the minimality for $\mathfrak n'=\mathfrak n-\Delta$ for any segment $\Delta$ in $\mathfrak n$. By Theorem \ref{thm convex derivatives} and Lemma \ref{lem cover consecutive}, it reduces to check that for any multisegment $\mathfrak m'$ obtained from $\mathfrak n'$ by an elementary intersection-union operations involving two consecutive segments, 
\[ D_{\mathfrak m'}(\pi)\not\cong D_{\mathfrak n'}(\pi). \]
 Denote such two consecutive segments by $\widetilde{\Delta}< \widetilde{\Delta}'$.


Let $\mathfrak m$ be obtained by the intersection-union process from $\mathfrak n$ involving $\widetilde{\Delta}$ and $\widetilde{\Delta}'$. Then $\mathfrak m'=\mathfrak m-\Delta$.

  We consider following possibilities:

\begin{enumerate}
\item Case 1: Suppose $\widetilde{\Delta}$ and $\widetilde{\Delta}'$ still form a pair of consecutive segments in $\mathfrak n$. Then we write the segments in $\mathfrak n$ as in Lemma \ref{lem cover consecutive} (with obvious notation replacement):
\[   \Delta_1, \ldots, \Delta_r, \widetilde{\Delta}, \widetilde{\Delta}', \Delta_1', \ldots, \Delta_s' .\]
\begin{enumerate}
  \item Case 1(a): $\Delta$ appears in one of $\Delta_1', \ldots, \Delta_s'$. If $D_{\mathfrak n-\Delta}(\pi)=D_{\mathfrak m-\Delta}(\pi)$, then the cancellative property (Proposition \ref{prop cancel property}) and Lemma \ref{lem cover consecutive}  (also see Proposition \ref{prop first subseq property}) imply that 
	\[ D_{\left\{ \Delta_1, \ldots, \Delta_r, \widetilde{\Delta}, \widetilde{\Delta}' \right\}}(\pi) \not\cong D_{\left\{ \Delta_1, \ldots, \Delta_r, \widetilde{\Delta}\cup \widetilde{\Delta}', \widetilde{\Delta}\cap \widetilde{\Delta}' \right\}}(\pi) .
	\]
	However, this implies $D_{\mathfrak n}(\pi)\not\cong D_{\mathfrak m}(\pi)$ by applying $D_{\Delta_1'}, \ldots, D_{\Delta_s'}$ with $D_{\Delta}$ omitted.
	\item Case 1(b): $\Delta$ appears in one of $\Delta_1, \ldots, \Delta_r$. Let 
	\[\mathfrak p=\left\{ \Delta_1, \ldots, \Delta_r, \widetilde{\Delta}, \widetilde{\Delta}' \right\}-\Delta \] 
and let $\mathfrak q$ be obtained from $\mathfrak p$ by an elementary intersection-union process on $\widetilde{\Delta}$ and $\widetilde{\Delta}'$. By the cancellative property (Proposition \ref{prop cancel property}), it suffices to show that $D_{\mathfrak p}(\pi)\not\cong D_{\mathfrak q}(\pi)$. 

Let $\tau=D_{\mathfrak p-\widetilde{\Delta}-\widetilde{\Delta}'-\Delta}(\pi)$. By repeatedly using Lemma \ref{lem min commut in special case}, we have:
\[ D_{\mathfrak p+\Delta}(\pi) \cong D_{\widetilde{\Delta}'}\circ D_{\widetilde{\Delta}}\circ D_{\Delta}(\tau) 
\]
and $\left\{\Delta, \widetilde{\Delta}, \widetilde{\Delta}' \right\}$ is minimal to $\tau$. Now by Lemma \ref{lem basic subsequent property}, we have that
\[ D_{\mathfrak p}(\pi) = D_{\widetilde{\Delta}'}\circ D_{\widetilde{\Delta}}(\tau)\not\cong  D_{\widetilde{\Delta}\cup \widetilde{\Delta}'}\circ D_{\widetilde{\Delta}\cap \widetilde{\Delta}'}(\tau)=D_{\mathfrak q}(\pi) 
\]
as desired.


\end{enumerate}
\item Case 2: $\widetilde{\Delta}$ and $\widetilde{\Delta}'$ do not form a consecutive pair. We then first write the segments in $\mathfrak n'$ as in Lemma \ref{lem cover consecutive}:
\[  \Delta_1, \ldots, \Delta_r, \widetilde{\Delta}, \widetilde{\Delta}', \Delta_1', \ldots, \Delta_s' .
\]
Since the pair is not consecutive, the segment $\Delta$ must take the form as in the second bullet of Definition \ref{def consecutive pair}. Then one can still check that
\[  \Delta_1, \ldots, \Delta_r, \widetilde{\Delta}, \Delta , \widetilde{\Delta}', \Delta_1', \ldots, \Delta_s'
\]
is an ascending sequence. Now let $\tau=D_{\left\{ \Delta_1, \ldots, \Delta_r \right\}}(\pi)$. By the cancellative property (Proposition \ref{prop cancel property}), it suffices to show that
\[ \quad D_{\left\{ \widetilde{\Delta}, \widetilde{\Delta}' \right\}}(\tau)\not\cong D_{\left\{ \widetilde{\Delta}\cap \widetilde{\Delta}', \widetilde{\Delta}\cup \widetilde{\Delta}' \right\}}(\tau) .
\]
This follows from the basic case of Lemma \ref{lem basic subsequent property 3}. 
\end{enumerate}
\end{proof}

\section{Commutativity and minimality} \label{s commutative minimal}

In this section, we study the commutativity for a minimal sequence. 


\subsection{Commutativity and minimality (anther basic case)}

\begin{lemma} \label{lem minmal commut 1}
Let $\pi \in \mathrm{Irr}_{\rho}$. Let $\mathfrak m \in \mathrm{Mult}_{\rho}$ be minimal to $\pi$. Let $\Delta \in \mathfrak m$. Then $D_{\mathfrak m-\Delta}\circ D_{\Delta}(\pi)\cong D_{\mathfrak m}(\pi)$.
\end{lemma}

\begin{proof}
 We write the segments in $\mathfrak m-\Delta$ in an ascending order: $\Delta_1, \ldots, \Delta_r$. By Proposition \ref{prop first subseq property}(1), we can reduce to the case that $\Delta_1, \ldots, \Delta_r, \Delta$ still form an ascending order. By Proposition \ref{prop first subseq property}(2) and the basic case (Proposition \ref{prop dagger property 2}), 
\[ D_{\mathfrak m}(\pi)\cong D_{\Delta} \circ D_{\mathfrak m-\Delta}(\pi) \cong D_{\Delta_r}\circ D_{\Delta}\circ D_{\mathfrak m-\Delta-\Delta_r}(\pi) .\]
By Theorem \ref{thm subsequent minimal}, $\mathfrak m-\Delta_r$ is still minimal to $\pi$. We now inductively obtain the statement.
\end{proof}

We first study a special case of commutativity and minimality, and we shall prove a full version in Theorem \ref{thm minimal and commut}.

\begin{lemma} \label{lem minimal commut}
Let $\pi \in \mathrm{Irr}_{\rho}$. Let $\mathfrak m \in \mathrm{Mult}_{\rho}$ be minimal to $\pi$. Let $c$ (resp. $d$) be the largest integer such that $\mathfrak m[ c ] \neq 0$ (resp. $\mathfrak m\langle d \rangle \neq 0$). Let $\Delta \in \mathfrak m[ c ]$ or $\in \mathfrak m\langle d \rangle$. Then $\mathfrak m-\Delta$ is minimal to $D_{\Delta}(\pi)$.
\end{lemma}

\begin{proof}
The condition in the lemma guarantees that $\Delta$ can be arranged in the last one for an ascending order for $\mathfrak m$.

 Suppose $\mathfrak m-\Delta$ is not minimal to $D_{\Delta}(\pi)$. Let $\mathfrak m'=\mathfrak m-\Delta$. By Theorem \ref{thm convex derivatives} and Lemma \ref{lem cover consecutive}, it suffices to show that for any pair $\widetilde{\Delta}< \widetilde{\Delta}'$ of consecutive segments,
\[ \mathfrak n':=\mathfrak m'-\left\{ \widetilde{\Delta}, \widetilde{\Delta}' \right\}+ \widetilde{\Delta}\cup \widetilde{\Delta}' + \widetilde{\Delta}\cap \widetilde{\Delta}'  \] 
does not give the same derivative on $D_{\Delta}(\pi)$ i.e.
\[ (*)\quad  D_{\mathfrak n'}\circ D_{\Delta}(\pi) \not\cong D_{\mathfrak m'}\circ D_{\Delta}(\pi) .\]
Now we arrange and relabel the segments in $\mathfrak m'$ as in Lemma \ref{lem cover consecutive}:
\[  \Delta_1, \ldots, \Delta_k, \widetilde{\Delta}, \widetilde{\Delta}' , \Delta_{k+3}, \ldots, \Delta_r ,
\]
which is in an ascending order. By Proposition \ref{prop first subseq property}, in order to show (*), it suffices to show that $\left\{ \Delta_1, \ldots, \Delta_k, \widetilde{\Delta}, \widetilde{\Delta}' \right\}$ is still minimal to $D_{\Delta}(\pi)$. Let 
\[ \tau =D_{\Delta_k}\circ \ldots \circ D_{\Delta_1}(\pi), \mbox{ and } \tau' =D_{\Delta_k}\circ \ldots D_{\Delta_1}\circ D_{\Delta}(\pi) . \] 
By Lemma \ref{lem intersect union ascending order} and Proposition \ref{prop cancel property}, it suffices to prove 
\begin{align} \label{eqn reduced iso in min comm}
   D_{\widetilde{\Delta}'}\circ D_{\widetilde{\Delta}}(\tau') \not\cong D_{\widetilde{\Delta}'\cup \widetilde{\Delta}}\circ D_{\widetilde{\Delta}'\cap \widetilde{\Delta}}(\tau') .
\end{align}

Let $\mathfrak p=\left\{ \Delta_1, \ldots, \Delta_k\right\}$. To this end, by the subsequent property (Theorem \ref{thm subsequent minimal}), we have that $\mathfrak p+\widetilde{\Delta}'+\widetilde{\Delta}+\Delta$ is minimal to $\pi$. Thus, we also have that $\left\{ \Delta, \widetilde{\Delta}', \widetilde{\Delta}\right\}$ is minimal to $\tau=D_{\mathfrak p}(\pi)$ by Proposition \ref{prop first subseq property}. Now, Lemma \ref{lem basic subsequent property 2} implies:
\[   D_{\widetilde{\Delta}'}\circ D_{\widetilde{\Delta}}\circ D_{\Delta}(\tau) \not\cong D_{\widetilde{\Delta}'\cup \widetilde{\Delta}}\circ D_{\widetilde{\Delta}'\cap \widetilde{\Delta}}\circ D_{\Delta}(\tau). 
\]
On the other hand, by the subsequent property (Theorem \ref{thm subsequent minimal}), it gives that $\left\{ \Delta_1, \ldots, \Delta_k , \Delta\right\}$ is minimal to $\pi$. Then combining with Lemma \ref{lem minmal commut 1}, we have $D_{\Delta}(\tau)\cong \tau'$. Combining, we have the desired non-isomorphism (\ref{eqn reduced iso in min comm}).
\end{proof}

\subsection{Commutativity and minimality (one segment case)}

\begin{theorem} \label{thm minimal and commut}
Let $\pi \in \mathrm{Irr}_{\rho}$. Let $\mathfrak n \in \mathrm{Mult}_{\rho}$ be minimal to $\pi$. Let $\Delta \in \mathfrak n$. Then $\mathfrak n-\Delta$ is minimal to $D_{\Delta}(\pi)$ and
\[  D_{\mathfrak n-\Delta}\circ D_{\Delta}(\pi) \cong D_{\mathfrak n}(\pi) .
\]
\end{theorem}

\begin{proof}
We first prove the second assertion. We write the segments in $\mathfrak n$ in an ascending order:
\[  \Delta_1, \ldots, \Delta_k, \Delta, \Delta_{k+1}, \ldots, \Delta_r \]
with $b(\Delta_1)\leq \ldots \leq b(\Delta)\leq \ldots \leq b(\Delta_r)$. Then, by Proposition \ref{prop first subseq property}, $\left\{\Delta_1, \ldots, \Delta_k, \Delta \right\}$ is still minimal to $\mathfrak n$. Then, the second assertion follows from Lemma \ref{lem minimal commut}. 

We now prove the first assertion. By Theorem \ref{thm convex derivatives} and Lemma \ref{lem cover consecutive}, it suffices to consider for a consecutive pair of segments. Let $\widetilde{\Delta}, \widetilde{\Delta}'$ be a pair of consecutive segments in $\mathfrak n-\Delta$. Let 
\[  \mathfrak n'=\mathfrak n-\Delta -\left\{ \widetilde{\Delta}, \widetilde{\Delta}' \right\}+ \widetilde{\Delta}\cup \widetilde{\Delta}'+ \widetilde{\Delta}\cap \widetilde{\Delta}' .
\]
We are going to show
\[  D_{\mathfrak n-\Delta}\circ D_{\Delta}(\pi)\not\cong D_{\mathfrak n'}\circ D_{\Delta}(\pi). 
\]

 We consider the following two cases:
\begin{enumerate}
\item $\widetilde{\Delta}$ and $\widetilde{\Delta}'$ still form a pair of consecutive segments in $\mathfrak n$. We arrange and relabel the segments as:
\[   \Delta_1, \ldots \Delta_p, \widetilde{\Delta}, \widetilde{\Delta}', \Delta_{p+1}, \ldots, \Delta_r 
\]
with $\Delta_{p+1}, \ldots, \Delta_r$ to be all the segments with $a(\Delta_t) \geq a(\widetilde{\Delta}')$ or $b(\Delta_t) \geq b(\widetilde{\Delta}')$. Similar to Lemma \ref{lem cover consecutive}, one has that
\[  \Delta_1, \ldots, \Delta_p, \widetilde{\Delta}\cup \widetilde{\Delta}', \widetilde{\Delta}\cap \widetilde{\Delta}', \Delta_{p+1}, \ldots, \Delta_r 
\]
form an ascending sequence. 

\begin{enumerate}
   \item[(i)] Suppose $\Delta$ appears in one of $\Delta_1, \ldots, \Delta_p$, say $\Delta_i$. Let $\tau=D_{\Delta_p}\circ \ldots \circ D_{\Delta_{i+1}}\circ D_{\Delta_{i-1}}\circ \ldots \circ D_{\Delta_1}\circ D_{\Delta}(\pi)$. In such case, the proved second assertion gives that
	\[  \tau \cong  D_{\Delta_p} \circ \ldots \circ D_{\Delta_1}(\pi) .
	\]
The isomorphism with the minimality of $\left\{ \Delta_1, \ldots, \Delta_p, \widetilde{\Delta}, \widetilde{\Delta}' \right\}$ to $\pi$ and the discussion on the ascending sequence above gives that
	\[  D_{\widetilde{\Delta}'}\circ D_{\widetilde{\Delta}}(\tau)\not\cong D_{\widetilde{\Delta}'\cup \widetilde{\Delta}}\circ D_{\widetilde{\Delta}'\cap \widetilde{\Delta}}(\tau) .
	\]
Applying $D_{\Delta_{p+1}}, \ldots, D_{\Delta_r}$, by Proposition \ref{prop cancel property}, we have that
	\[   D_{\mathfrak n-\Delta}\circ D_{\Delta}(\pi)\not\cong D_{\mathfrak n'}\circ D_{\Delta}(\pi).
	\]
	 \item[(ii)] Suppose $\Delta$ appears in one of $\Delta_{p+1}, \ldots, \Delta_r$, say $\Delta_j$. By rearranging and relabeling the segments in $\Delta_{p+1}, \ldots, \Delta_r$ if necessary, we assume $b(\Delta_{p+1})\leq \ldots \leq b(\Delta_r)$ if $b(\widetilde{\Delta}') \leq b(\Delta_j)$ and assume $a(\Delta_{p+1})\leq \ldots \leq a(\Delta_r)$ if $a(\widetilde{\Delta}')\leq a(\Delta_j)$. Let
\[  \mathfrak n_j=\left\{ \Delta_1, \ldots, \Delta_p, \widetilde{\Delta}, \widetilde{\Delta}', \Delta_{p+1}, \ldots, \Delta_j=\Delta \right\} ,\]
\[  \mathfrak n_j'=\left\{ \Delta_1, \ldots, \Delta_p, \widetilde{\Delta}\cup \widetilde{\Delta}', \widetilde{\Delta}\cap \widetilde{\Delta}', \Delta_{p+1}, \ldots, \Delta_j=\Delta \right\} .
\]
	Then, by Lemma \ref{lem minimal commut}, 
\[    D_{\mathfrak n_j'-\Delta} \circ D_{\Delta}(\pi) \not\cong D_{\mathfrak n_j-\Delta}\circ D_{\Delta}(\pi) 
\]
Thus, applying the derivatives $D_{\Delta_{j+1}},\ldots, D_{\Delta_r}$, we have $D_{\mathfrak n'}\circ D_{\Delta}(\pi)\not\cong D_{\mathfrak n-\Delta}\circ D_{\Delta}(\pi)$ as desired. 
\end{enumerate}
\item $\widetilde{\Delta}$ and $\widetilde{\Delta}'$ do not form a consecutive pair in $\mathfrak n$. One uses similar argument in Theorem \ref{thm subsequent minimal} to reduce to three segment case. Then one reduces to a basic case (see similar discussions in the proof of Theorem \ref{thm subsequent minimal}), that is Lemma \ref{lem minimal in a basic case}. In detail, we arrange the segments in $\mathfrak n$ in an ascending order:
\[  \Delta_1, \ldots, \Delta_p, \widetilde{\Delta}, \Delta, \widetilde{\Delta}', \Delta_{p+1}, \ldots, \Delta_r .
\]
Then, we still have
\[  \Delta_1, \ldots, \Delta_p, \widetilde{\Delta}, \Delta, \widetilde{\Delta}'
\]
is still minimal to $\pi$ by Theorem \ref{thm subsequent minimal}. Then we also have $\left\{ \widetilde{\Delta}, \Delta, \widetilde{\Delta}'\right\}$ is minimal to 
\begin{align} \label{eqn intersection non consec comm}
\tau':= D_{\Delta_p} \circ \ldots \circ D_{\Delta_1}(\pi) .
\end{align}
By Lemma \ref{lem minimal in a basic case}, we then have that 
\[  D_{\widetilde{\Delta}'}\circ D_{\widetilde{\Delta}} \circ D_{\Delta}(\tau') \not\cong D_{\widetilde{\Delta}'\cup \widetilde{\Delta}}\circ D_{\widetilde{\Delta}'\cap \widetilde{\Delta}} \circ D_{\Delta}(\tau') .
\]
But now, by Theorem \ref{thm subsequent minimal}, $\left\{ \Delta_1, \ldots, \Delta_p, \Delta\right\}$ is minimal to $\pi$, and so by the proved second assertion, we also have 
\begin{align} \label{eqn intersect 2}
\tau' \cong D_{\Delta_p}\circ \ldots \circ D_{\Delta_1}\circ D_{\Delta}(\pi) .   
\end{align}
Now, one applies $D_{\Delta_{p+1}}, \ldots, D_{\Delta_r}$ on (\ref{eqn intersection non consec comm}), and then uses (\ref{eqn intersect 2}) and Proposition \ref{prop cancel property} to see that $D_{\mathfrak n-\Delta} \circ D_{\Delta}(\pi) \not\cong D_{\mathfrak n'}\circ D_{\Delta}(\pi)$ as desired.

\end{enumerate}

\end{proof}

\subsection{General form of commutativity and minimality}

\begin{theorem} \label{thm minimal and commut 2}
Let $\pi \in \mathrm{Irr}_{\rho}$. Let $\mathfrak n \in \mathrm{Mult}_{\rho}$ be minimal to $\pi$. Let $\mathfrak n'$ be a submultisegment of $\mathfrak n$. Then $D_{\mathfrak n-\mathfrak n'}\circ D_{\mathfrak n'}(\pi)\cong D_{\mathfrak n}(\pi)$ and $\mathfrak n-\mathfrak n'$ is minimal to $D_{\mathfrak n'}(\pi)$. 
\end{theorem}

\begin{proof}
This follows by repeatedly using Theorem \ref{thm minimal and commut}.
\end{proof}

\part{Representation-theoretic aspects} \label{part rep theoy asp}

\section{$\eta$-invariant and commutativity} \label{s eta invariant and commutativity}

We shall first discuss the representation-theoretic interpretation for $\eta$-invariants. Then we explain how to relate with the commutativity. The representation-theoretic approach provides different techniques, which are of independent interests.

\subsection{Representation-theoretic counterpart of $\eta_{\Delta}$}

Let $\Delta=[a,b]_{\rho}$. For $\pi \in \mathrm{Irr}_{\rho}$, let
\begin{eqnarray} \label{eqn mx segment}
 \mathfrak{mx}(\pi, \Delta) = \sum_{a\leq a'\leq b} \varepsilon_{[a',b]_{\rho}}(\pi)\cdot [a',b]_{\rho} , 
\end{eqnarray}
which means that $[a',b]_{\rho}$ appears with multiplicity $\varepsilon_{[a',b]_{\rho}}(\pi)$ in $\mathfrak{mx}(\pi, \Delta)$. This is the multisegment analogue of the $\eta$-invariant.

The following is the key property:

\begin{lemma} \cite[Proposition 11.1]{Ch22+} \label{lem eta invariant rep}
Let $l=l_{abs}(\mathfrak{mx}(\pi, \Delta))$. Let $\mathfrak m=\mathfrak{mx}(\pi, \Delta)$. Then $D_{\mathfrak m}(\pi)\boxtimes \mathrm{St}(\mathfrak m)$ is a direct summand in $\pi_{N_l}$.
\end{lemma}

\subsection{Commutativity}

Here we explain how to view the commutativity of Proposition \ref{prop dagger property 2} from the perspective of Lemma \ref{lem eta invariant rep}. We first show the following new lemma, using Lemma \ref{lem eta invariant rep}:

\begin{lemma} \label{lem commut for max case}
Let $\pi \in \mathrm{Irr}_{\rho}(G_n)$. Let $\Delta_1, \Delta_2 \in \mathrm{Seg}_{\rho}$ be admissible to $\pi$. Suppose $\Delta_1<\Delta_2$. Suppose $(\Delta_1, \Delta_2, \pi)$ satisfies the non-overlapping property. Let $\mathfrak m=\mathfrak{mx}(\pi, \Delta_2)$.  Then $D_{\Delta_1}\circ D_{\mathfrak m}(\pi)\cong D_{\mathfrak m} \circ D_{\Delta_1}(\pi)$.
\end{lemma}

\begin{proof}
Let $l=l_{abs}(\Delta_1)$ and let $N=N_l$. We have that:
\[   \pi \hookrightarrow D_{\mathfrak m}(\pi)\times \mathrm{St}(\mathfrak m) .\]
Then 
\[  D_{\Delta_1}(\pi)\boxtimes \mathrm{St}(\Delta_1) \hookrightarrow \pi_N \hookrightarrow (D_{\mathfrak m}(\pi) \times \mathrm{St}(\mathfrak m))_N .\]
An anslysis on the layers in the geometric lemma, we have that
\[ (*)\quad D_{\Delta_1}(\pi) \boxtimes \mathrm{St}(\Delta_1) \hookrightarrow D_{\mathfrak m}(\pi)_{N_l} \dot{\times}^1 \mathrm{St}(\mathfrak m) ,
\]
where $\dot{\times}^1$ means that the induction to a $G_{n-l}\times G_l$-representation. By Proposition \ref{prop dagger property} and Lemma \ref{lem eta invariant rep}, $D_{\mathfrak m}\circ D_{\Delta_1}(\pi)\boxtimes \mathrm{St}(\mathfrak m)$ is a direct summand in $D_{\Delta_1}(\pi)_{N_{l'}}$, where $l'=l_{abs}(\mathfrak m)$. Furthermore, no other composition factors in $D_{\Delta_1}(\pi)_{N_{l'}}$take the form $\tau \boxtimes \mathrm{St}(\mathfrak m)$. Now, via Frobenius reciprocity on the map in (*), we obtain a non-zero map from  $D_{\Delta_1}(\pi)_{N_{l'}} \boxtimes \mathrm{St}(\Delta_1)$ to $(D_{\mathfrak m}(\pi)_{N_l} \boxtimes \mathrm{St}(\mathfrak m))^{\phi}$, where $\phi$ is a twisting sending a $G_{n-l-l'}\times G_l\times G_{l'}$-representation to a $G_{n-l-l'}\times G_{l'}\times G_l$-representation. Then
\[  D_{\mathfrak m}\circ D_{\Delta_1}(\pi) \boxtimes \mathrm{St}(\Delta_1) \hookrightarrow D_{\mathfrak m}(\pi)_{N_l} .
\]
Thus, we have $D_{\mathfrak m}\circ D_{\Delta_1}(\pi)\cong D_{\Delta_1}\circ D_{\mathfrak m}(\pi)$. 
\end{proof}

Lemma \ref{lem commut for max case} can also be deduced from Proposition \ref{prop dagger property 2} and Lemma \ref{lem dagger property on derivative}. On the other hand, one can also give another proof for Proposition \ref{prop dagger property 2} by using Lemmas \ref{lem dagger property on derivative} and \ref{lem commut for max case}.

\section{Conjectural interpretation for minimal sequences} \label{s conj model minimal}

\subsection{Minimality for two segments}

\begin{proposition}   \label{prop injective for a pair}
Let $\pi \in \mathrm{Irr}_{\rho}$. Let $\Delta_1, \Delta_2$ be a pair of linked segments with $\Delta_1<\Delta_2$. Suppose 
\[ D_{\Delta_1\cup \Delta_2}\circ D_{\Delta_1\cap \Delta_2}(\pi) \not\cong D_{\Delta_2}\circ D_{\Delta_1}(\pi). \]
 Then the unique non-zero map 
\[    D_{\Delta_2}\circ D_{\Delta_1}(\pi) \boxtimes (\mathrm{St}(\Delta_1)\times \mathrm{St}(\Delta_2)) \rightarrow \pi_N ,
\]
where $N=N_{l_{abs}(\Delta_1)+l_{abs}(\Delta_2)}$, is injective.
\end{proposition}

\begin{proof}
Suppose the map is not injective. Since $D_{\Delta_2}\circ D_{\Delta_1}(\pi) \boxtimes (\mathrm{St}(\Delta_1)\times \mathrm{St}(\Delta_2))$ is indecomposable and has length $2$, the image of the map can only be isomorphic to $D_{\Delta_2}\circ D_{\Delta_1}(\pi)\boxtimes (\mathrm{St}(\Delta_1\cup \Delta_2)\times \mathrm{St}(\Delta_1\cap \Delta_2))$. This implies that 
\[ D_{\Delta_2}\circ D_{\Delta_1}(\pi)\boxtimes (\mathrm{St}(\Delta_1\cup \Delta_2)\times \mathrm{St}(\Delta_1\cap \Delta_2)) \]
is a submodule of $\pi_N$. Then, applying Frobenius reciprocity, we have that $\pi$ is the unique submodule of $D_{\Delta_2}\circ D_{\Delta_1}(\pi)\times \mathrm{St}(\Delta_1\cup \Delta_2+\Delta_1\cap \Delta_2)$ (see \cite{LM16, Ch22+}). 

Recall that $\mathrm{St}(\Delta_1 \cup \Delta_2+\Delta_1 \cap \Delta_2)\cong \mathrm{St}(\Delta_1\cup \Delta_2) \times \mathrm{St}(\Delta_1\cap \Delta_2)$. Let $\tau$ be the unique submodule of $D_{\Delta_2}\circ D_{\Delta_1}(\pi) \times \mathrm{St}(\Delta_1 \cup \Delta_2)$ so that 
\[ (*)\quad  D_{\Delta_1\cup \Delta_2}(\tau) \cong D_{\Delta_2}\circ D_{\Delta_1}(\pi) .\]
 Then the uniqueness of submodule above also forces that
\[   \pi \hookrightarrow \tau \times \mathrm{St}(\Delta_1\cap \Delta_2) 
\]
and so $D_{\Delta_1\cap \Delta_2}(\pi) \cong \tau$. Combinbing with (*), we have:
\[ D_{\Delta_1\cup \Delta_2}\circ D_{\Delta_1\cap \Delta_2}(\pi) \cong D_{\Delta_2}\circ D_{\Delta_1}(\pi),\]
 giving a contradiction.
\end{proof}

In the Appendix (Section \ref{s lemma on non-isom derivatives}), we shall prove a converse of Proposition \ref{prop injective for a pair}.

\subsection{A representation-theoretic interpretation of minimal sequences}\begin{definition}
For a multisegment $\mathfrak h=\left\{ \Delta_1, \ldots, \Delta_r \right\} \in \mathrm{Mult}_{\rho}$ labelled in an ascending order, define
\[  \widetilde{\lambda}(\mathfrak h) := \mathrm{St}(\Delta_1)\times \ldots \times \mathrm{St}(\Delta_r) . \]
We shall call it a {\it co-standard representation}. Sometimes $\lambda(\mathfrak h)$ is used for standard modules and so we prefer to use $\widetilde{\lambda}$ here.
\end{definition}

We conjecture that Proposition \ref{prop injective for a pair} can be generalized to general minimal multisegments.

\begin{conjecture} \label{conj minimal model}
Let $\pi \in \mathrm{Irr}_{\rho}$. Let $\mathfrak n \in \mathrm{Mult}_{\rho}$ be minimal to $\pi$.  Then the unique non-zero map
\[  D_{\mathfrak n}(\pi) \boxtimes \widetilde{\lambda}(\mathfrak n) \rightarrow \pi_{N_l}, 
\] 
where $l=l_{abs}(\Delta_1)+\ldots +l_{abs}(\Delta_r)$, is injective.
\end{conjecture}

We remark that the uniqueness of the non-zero map in Conjecture \ref{conj minimal model} follows from the uniqueness of simple quotients of Bernstein-Zelevinsky derivatives, shown in \cite[Proposition 3.15]{Ch22+d}.

\section{Some applications on the embedding model} \label{s application embedding models}

In this section, we shall discuss applications of the embedding model arising from the minimality in Proposition \ref{prop injective for a pair}.




\begin{lemma}\label{lem strategy for checking minimal}
Let $\Delta', \Delta'', \Delta''' \in \mathrm{Seg}_{\rho}$. Suppose $\Delta'<\Delta''$. Let $\tau=D_{\Delta''}\circ D_{\Delta'}(\pi)$ and let $l'=l_{abs}(\Delta'), l''=l_{abs}(\Delta''), l'''=l_{abs}(\Delta''')$. Let $\kappa=D_{\Delta'''}(\pi)$. Suppose the followings hold:
\begin{itemize}
\item $\mathrm{dim}~\mathrm{Hom}_{G_{n-l'-l''}\times G_{l'+l''}}(\tau \boxtimes (\mathrm{St}(\Delta')\times \mathrm{St}(\Delta'')), (\kappa \times \mathrm{St}(\Delta'''))_{N_{l'+l''}}) \leq 1$;
\item The non-zero map in the first bullet factors through the natural embedding:
\[   \kappa_{N_{l'+l''}} \dot{\times}^1 \mathrm{St}(\Delta''') \hookrightarrow (\kappa \times \mathrm{St}(\Delta'''))_{N_{l'+l''}} 
\]
from the bottom layer in the geometric lemma. Here $\dot{\times}^1$ again denotes a parabolic induction from a $G_{n-l'-l''-l'''}\times G_{l'''} \times G_{l'+l''} $-representation to a $G_{n-l'-l''}\times G_{l'+l''}$-representation.
\item $D_{\Delta'''}\circ D_{\Delta''}\circ D_{\Delta'}(\pi)\cong D_{\Delta''}\circ D_{\Delta'''}\circ D_{\Delta'}(\pi)$. 
\end{itemize}
Then, if $\left\{ \Delta', \Delta'' \right\}$ is minimal to $\pi$, then $\left\{ \Delta', \Delta'' \right\}$ is minimal to $D_{\Delta'''}(\pi)$.
\end{lemma}

\begin{proof}
Suppose $\left\{ \Delta', \Delta'' \right\}$ is not minimal to $D_{\Delta'''}(\pi)$. Let $i=l'+l''$. Let $\lambda_1=\mathrm{St}(\Delta'\cup \Delta'')\times \mathrm{St}(\Delta' \cap \Delta'')$. Let $i=l_{abs}(\Delta')+l_{abs}(\Delta'')$. Then, by Proposition \ref{prop injective for a pair},
\[            D_{\Delta''}\circ D_{\Delta'}\circ D_{\Delta'''}(\pi)  \boxtimes  \lambda_1 \hookrightarrow  D_{\Delta'''}(\pi)_{N_{i}} .
\]

On the other hand, we have the following embedding:
\[   \pi \hookrightarrow D_{\Delta'''}(\pi)\times \mathrm{St}(\Delta''') .
\]
Let $\omega=D_{\Delta''}\circ D_{\Delta'}\circ D_{\Delta'''}(\pi) \times \mathrm{St}(\Delta''')$. Via the bottom layer in the geometric lemma, we have an embedding:
\[  \omega \boxtimes \lambda_1 \hookrightarrow (D_{\Delta'''}(\pi)\times \mathrm{St}(\Delta''') )_{N_i}.
\]
By the third bullet of the hypothesis, $D_{\Delta''}\circ D_{\Delta'}(\pi)$ is a submodule of $\omega$. Combining above, we have that:
\[  D_{\Delta''}\circ D_{\Delta'}(\pi)\boxtimes \lambda_1 \hookrightarrow (D_{\Delta'''}(\pi)\times \mathrm{St}(\Delta''') )_{N_i}.
\]

 Let $\lambda_2=\mathrm{St}(\Delta')\times \mathrm{St}(\Delta'')$. Then, by the minimality of $\left\{ \Delta', \Delta'' \right\}$ to $\pi$ and Proposition \ref{prop injective for a pair}, we have 
\[  D_{\Delta''}\circ D_{\Delta'}(\pi)\boxtimes \lambda_2 \hookrightarrow \pi_{N_{i}} .
\]
This induces another embedding: 
\[ D_{\Delta''}\circ D_{\Delta'}(\pi) \boxtimes \lambda_2 \hookrightarrow ( D_{\Delta'''}(\pi) \times \mathrm{St}(\Delta'''))_{N_{i}}. \]
Since $\lambda_1$ and $\lambda_2$ have no isomorphic submodules, the above two embeddings give:
\[D_{\Delta''}\circ D_{\Delta'}(\pi)\boxtimes (\lambda_1 \oplus \lambda_2)\hookrightarrow (D_{\Delta'''}(\pi) \times \mathrm{St}(\Delta'''))_{N_{i}} .
\]
However, this induces two non-zero maps from $D_{\Delta''}\circ D_{\Delta'}(\pi)\boxtimes \lambda_1$ to $(D_{\Delta'''}(\pi) \times \mathrm{St}(\Delta'''))_{N_i}$, which are not scalar multiple of each other. This contradicts to the first bullet.
\end{proof}

Lemma \ref{lem strategy for checking minimal} provides another strategy for checking minimality in some three segment cases in Section \ref{s three segment cases} e.g. the second bullet case in the proof of Lemma \ref{lem minimal in a basic case} and the linked case in the proof of Lemma \ref{lem basic subsequent property 2}. Checking the second bullet usually involves analysis on the layers arising from the geometric lemma while checking the third bullet usually uses some known commutativity from minimality (Proposition \ref{prop dagger property 2}) in some other cases. Checking the first bullet requires some inputs of multiplicity-one theorems from \cite{AGRS10} and \cite{Ch21+}.

\section{Embedding model, minimality and removal process} \label{s embedding minimality}

\subsection{Combintorial preliminaries}
Let $S_n$ be the symmetric group permuting the integers $\left\{1, \ldots, n \right\}$. 

\begin{definition}
Let $w \in S_n$. For $1\leq k \leq n$ and $1 \leq l \leq n$, define
\[   w[k,l] := |\left\{ a : 1 \leq a \leq k, w(a) \geq l  \right\}| .\]
\end{definition}
We shall write $\leq_B$ to be the Bruhat ordering on $S_n$ i.e. $w' \leq_B w$ if and only if $w'$ is a subword of a reduced expression of $w$. We write $w'<_B w$ if $w' \leq_B w$ and $w'\neq w$.

\begin{proposition} \cite[Theorem 2.1.5]{BB05} \label{prop equiv bruhat order Sn}
Let $w, w' \in S_n$. Then the following statements are equivalent:
\begin{enumerate}
\item $w' \leq_B w$;
\item $w'[k,l] \leq w[k,l]$ for any $k,l$.
\end{enumerate}

\end{proposition}

We shall now discuss a special situation. Let $S^{n-i,i}$ be the set of minimal representatives in the cosets in $ S_n/(S_{n-i}\times S_i)$. It is well-known that if $w \in S^{n-i,i}$, then
\begin{itemize}
\item $w(k)<w(l)$ if $k,l \in \left\{ 1, \ldots, n-i\right\}$ and $k<l$;
\item $w(k)<w(l)$ if $k,l \in \left\{ n-i+1, \ldots, n \right\}$ and $k<l$.
\end{itemize}

It is straightforward to obtain the following proposition from Proposition \ref{prop equiv bruhat order Sn}:

\begin{proposition} \label{prop criteria for checking bruhat}
Let $w, w' \in S^{n-i, i}$. The following statements are equivalent:
\begin{enumerate}
\item $w' \leq_B w$;
\item $w'(k) \leq w(k)$ for all $k =1, \ldots, n-i$;
\item $w'(k) \geq w(k)$ for all $k=n-i+1, \ldots, n$.
\end{enumerate}
\end{proposition}

\subsection{Embedding model and removal process}

We now explain how the removal process plays a role in the embedding model. We first provide a basic case in Proposition \ref{prop removal embedding} and then conjecture more general case in Section \ref{conj multisegment}.

\begin{proposition} \label{prop removal embedding}
Let $\mathfrak h \in \mathrm{Mult}_{\rho}$. Let $l=l_{abs}(\Delta)$. Then there exists an embedding
\[  \widetilde{\lambda}(\mathfrak r(\Delta, \mathfrak h)) \boxtimes \mathrm{St}(\Delta) \hookrightarrow \widetilde{\lambda}(\mathfrak h)_{N_l} .
\]
\end{proposition}

\begin{proof}

Write the segments in $\mathfrak m$ as $\Delta_i=[a_i,b_i]_{\rho}$ ($i=1, \ldots, r$). We arrange the segments in $\mathfrak m$ satisfying:
\begin{itemize}
\item $b_1 \leq b_2 \leq \ldots \leq b_r$;
\item if $b_i=b_{i+1}$, then $a_{i+1}\geq a_i$. 
\end{itemize}

We now apply the geometric lemma on $\widetilde{\lambda}(\mathfrak m)_{N_l}$. We first write the segments as:
\[  \Delta_1=[a_1,b_1]_{\rho}, \ldots , \Delta_r=[a_r, b_r]_{\rho} .
\]
and $\Delta=[a,b]_{\rho}$. Set 
\[  \Delta_k^{+,l_k}=[a_k+l_k, b_k]_{\rho}, \quad \Delta_k^{-,l_k}=[a_k, a_k+l_k-1]_{\rho}
\]
Those are possiby empty sets.

Then the layers arising from the geometric lemma takes the form:
\[ (*)\quad   (\mathrm{St}(\Delta_1^{+,l_1}) \times \ldots \times \mathrm{St}(\Delta_r^{+,l_r}))\boxtimes (\mathrm{St}(\Delta_1^{-,l_1})\times \ldots \times \mathrm{St}(\Delta_r^{-,l_r}) ,
\]
where $l_1, \ldots, l_r$ run for all integers such that $l_1+\ldots +l_r=l/\mathrm{deg}(\rho)$. Let $d=\mathrm{deg}(\rho)$. Let $t_k=l_{abs}(\Delta_k^{+,l_k})$. We now describe the underlying element in $(S_{t_1}\times \ldots \times S_{t_r})\setminus S_n /(S_{n-i}\times S_i)$ corresponding to the geometric lemma. The assignment takes the form:
\[   n-(j-1) \mapsto t_1+\ldots + t_{r-1}+t_r -(j-1)\quad \mbox{ for $j=1, \ldots, l_rd$ } , \]
\[   (n-l_rd)-(j-1) \mapsto t_1+\ldots +t_{r-1}-(j-1) \quad \mbox{ for $j=1, \ldots, l_{r-1}d$} , \]
\[  \vdots  \]
\[   (n-l_2d-\ldots-l_rd) -(j-1) \mapsto t_1-(j-1) \quad \mbox{ for $j=1, \ldots, l_1d$ }. \]
The assignment for $x <n-l_1d-l_2d-\ldots-l_rd$ is then uniquely determined by using the properties of elements in $S_n/(S_{n-i}\times S_i)$ stated before Proposition \ref{prop equiv bruhat order Sn}. 

We now consider a specific layer from the geometric lemma. The segments are chosen in the order as the removal sequence for $(\Delta, \mathfrak h)$:
\begin{eqnarray} \label{eqn removal sequence associate}
   \Delta_{p_1}, \ldots , \Delta_{p_e} .
\end{eqnarray}
For those indexes, if there is more than one segment for $\Delta_{p_i}$, we always choose the one of smaller index $p_i$. By using the nesting property and our arrangement of segments in $\mathfrak m$, we have that 
\[   p_1 >p_2 >\ldots > p_e .
\]
If $k\neq p_i$ for some $i=1, \ldots, e$, then we set $\widetilde{l}_k=0$. If $k=p_i$ for some $i=1, \ldots, e-1$, then we set $\widetilde{l}_k=a_{p_{i-1}}-a_{p_i}$. If $k=p_e$, then we set $\widetilde{l}_{k}=b-a_{p_e}$.

We now study layers of the geometric lemma of the form (*) such that, as sets, 
\[  \Delta_1^{+, l_1} \cup \ldots \cup \Delta_r^{+,l_r} = \Delta .
\]
For such layer, we say it is in {\it standard order} if 
\[  a(\Delta_1^{+,l_x})  \geq a(\Delta_r^{+,l_y}) \]
for any $x < y$ with $l_x, l_y \neq 0$.

We now analyse some behaviors of two following cases.
\begin{enumerate}
\item Case 1: $\Delta_1^{+,l_1}, \ldots, \Delta_r^{+, l_r}$ are in standard order. Let $q_1< \ldots< q_k$ be all the indexes such that $l_{q_x}\neq 0$. We first prove the following claim: \\

\noindent
{\it Claim 1:} The sequence $\Delta_{q_1}, \ldots, \Delta_{q_k}$ satisfies the nesting property i.e.
\[   \Delta_{q_1} \supset \ldots \supset \Delta_{q_k} . \]

\noindent
{\it Proof of Claim 1:} Suppose the sequence does not satisfy the nesting property. Since the nesting property is not transitive, we have that a pair $\Delta_{q_x}, \Delta_{q_{x+1}}$ does not satisfy the nesting property. Due to the arrangement of the segments in $\mathfrak m$, we must have that $b_{q_x} <b_{q_{x+1}}$ and so the violation of the nestng property implies that 
\[ a_{q_x}<a_{q_{x+1}}.  \]
 However, this then contradicts that the sequence is in standard order.\\

As a consequence of Claim 1, we also have that $l_{q_i}=a_{q_{i-1}}-a_{q_i}$.

Let $w^*$ be the fixed  element in $(S_{t_1}\times \ldots \times S_{t_r})\setminus S_n/(S_{n-i}\times S_i)$ associated to (\ref{eqn removal sequence associate}) and let $w$ be the element in $(S_{t_1}\times \ldots \times S_{t_r})\setminus S_n/(S_{n-i}\times S_i)$ associated to any layer in standard order in the sense defined above with $w\neq w^*$. \\

\noindent
{\it Claim 2:} $ w \not \leq_B w^* $. 

\noindent
{\it Proof of Claim 2:} The strategy is to apply Proposition \ref{prop criteria for checking bruhat}. Suppose not to derive a contradiction. Let $i^*$ be the largest integer such that $p_{i^*}=q_{i^*}$. Set $i^*=0$ if such integer does not exist i.e. $p_1 \neq q_1$. 

If $i_*=r$, then $q_{i^*+1}$ is also not defined. Otherwise, by the nesting property shown in Claim 1, we obtain a contradiction to choices of segments in the removal sequence for $(\Delta, \mathfrak h)$. This implies that $w^*=w$, giving a contradiction.

Thus $i^*<r$. Now we compare $\Delta_{p_{i^*+1}}$ and $\Delta_{q_{i^*+1}}$. By using Definition \ref{def removal process}(2) and Claim 1, we must have that 
\[   a_{p_{i^*+1}} \geq a_{q_{i^*+1}} , \]
or $q_{i^*+1}$ is not defined.

We further divide into two subcases. 
\begin{itemize}
\item $a_{p_{i^*}+1}=a_{q_{i^*+1}}$. But now, we must have that $b_{q_{i^*+1}} \geq b_{p^*+1}$ by the nesting property in Claim 1 again. Then, from our arrangements and choices, $q_{i^*+1} >p_{i^*+1}$. As noted from above that $l_{q_1}=\widetilde{l}_{p_1}, \ldots, l_{q_{i^*}}=\widetilde{l}_{p_{i^*}}$, one checks that 
\[ w(n-l_1d-\ldots-l_{i^*}d-1)>w^*(n-l_1d-\ldots-l_{i^*}d-1) \]
 and so $w \not\leq_B w^*$ by Proposition \ref{prop criteria for checking bruhat}. 
\item $a_{p_{i^*}+1}>a_{q_{i^*+1}}$ or $q_{i^*+1}$ is not defined. In such case, $l_{q_{i^*}} > \widetilde{l}_{p_{i^*}}$. Hence we also have that:
\[   w(n-l_1d-\ldots-l_{i^*}d-1)>w^*(n-l_1d-\ldots-l_{i^*}d-1) 
\]
and so again $w \not\leq_B w^*$ by Proposition \ref{prop criteria for checking bruhat}. 
\end{itemize}

\item Case 2: $\Delta_1^{+, l_1}, \ldots, \Delta_r^{+,l_r}$ are not in standard order. 

Let $\omega_{l_1, \ldots, l_r}=\mathrm{St}(\Delta_1^{-,l_1}) \times \ldots \times \mathrm{St}(\Delta_r^{-,l_r})$ and similarly let $\omega_{\widetilde{l}_1, \ldots, \widetilde{l}_r}$. \\

\noindent
{\it Claim 3:} Let $\widetilde{l}=\widetilde{l}_1+\ldots+\widetilde{l}_r$. For all $j$, 
\[  \mathrm{Ext}^j_{G_{n-\widetilde{l}}\times G_{\widetilde{l}}}( \omega_{\widetilde{l}_1, \ldots, \widetilde{l}_r}\boxtimes \mathrm{St}(\Delta), \omega_{l_1, \ldots, l_r}\boxtimes \mathrm{St}(\Delta_1^{+,l_1})\times \ldots \times \mathrm{St}(\Delta_r^{+,l_r}))=0 
\]

\noindent
{\it Proof of Claim 3:} We apply Frobenius reciprocity on the second factor. Then the Jacquet module of $\mathrm{St}(\Delta)$ takes the form
\[  \mathrm{St}(\Delta'_1) \boxtimes \ldots \boxtimes \mathrm{St}(\Delta'_r)
\]
with $\Delta_1', \ldots, \Delta_r'$ in standard order. Since $\Delta_1^{+,l_1}, \ldots, \Delta_r^{+,l_r}$ are not in standard order, an argument on comparing cuspidal support gives 
\[ \mathrm{Ext}^j_{G_{\widetilde{l}}}(\mathrm{St}(\Delta), \mathrm{St}(\Delta'_1) \boxtimes \ldots \boxtimes \mathrm{St}(\Delta'_r))=0 \]
for all $j$. Then K\"unneth formula then gives the claim.

\end{enumerate}

We now go back to the proof. The element $w^* \in (S_{t_1}\times \ldots \times S_{t_r})\setminus S_n / (S_{n-i}\times S_i)$ is defined as above. For each $w \in  (S_{t_1}\times \ldots \times S_{t_r})\setminus S_n / (S_{n-i}\times S_i)$, let $\kappa(w)$ be the associated layer taking the form (*). Note that $\omega_{\widetilde{l}_1, \ldots, \widetilde{l}_r}  \cong \widetilde{\lambda}(\mathfrak r(\Delta, \mathfrak h))$.

 Let $\kappa_{\leq w^*}$ (resp. $\kappa_{< w^*}$) be the submodule of $\widetilde{\lambda}(\mathfrak h)_{N_{l}}$ that consists of only the layers associated to $w' \in  (S_{t_1}\times \ldots \times S_{t_r})\setminus S_n / (S_{n-i}\times S_i)$ satisfying $w' \leq_B w^*$ (resp. $w' <_B w^*$). We have the short exact sequence:
\[   0 \rightarrow  \kappa_{< w^*} \rightarrow \kappa_{\leq w^*} \rightarrow \kappa(w^*) \rightarrow 0 . \]

On the other hand, there is an embedding
\[    \omega_{\widetilde{l}_1, \ldots, \widetilde{l}_r}   \boxtimes \mathrm{St}(\Delta) \hookrightarrow   \kappa(w^*)
\]
Hence, we have a submodule $\kappa'$ of $\kappa_{\leq w^*}$ admitting a short exact sequence:
\[    0 \rightarrow  \kappa_{< w^*}   \rightarrow \kappa' \rightarrow  \omega_{\widetilde{l}_1, \ldots, \widetilde{l}_r}  \boxtimes \mathrm{St}(\Delta) \rightarrow 0.
 \]
However,  for $w' <w^*$, we can conclude that 
\[   \mathrm{Ext}^j_{G_{n-l}\times G_l}( (\mathfrak r(\Delta, \mathfrak h))  \boxtimes \mathrm{St}(\Delta) ,\kappa(w')) = 0
\]
for all $j$. This follows from a standard cuspdial support argument if the $G_l$-part of $\kappa(w')$ does not have the same cuspidal support as $\mathrm{St}(\Delta)$, and follows from Claim 2 and Claim 3 otherwise. In other words, we have:
\[   \kappa' \cong \kappa_{<w^*} \oplus \widetilde{\lambda}(\mathfrak r(\Delta, \mathfrak h))  \boxtimes \mathrm{St}(\Delta) .
\]
This then gives the following desired embedding
\[   \widetilde{\lambda}(\mathfrak r(\Delta, \mathfrak h)) \boxtimes \mathrm{St}(\Delta) \hookrightarrow \kappa' \hookrightarrow \kappa \hookrightarrow \lambda(\widetilde{\mathfrak h})_{N_{l}}.
\]

\end{proof}

\subsection{Conjectures} \label{ss conjectures}

For $n_1+\ldots+n_s=n$, define $P_{n_1, \ldots, n_s}$ to be the parabolic subgroup of $G_n$ generated by the matrices $\mathrm{diag}(g_1, \ldots, g_s)$ (each $g_i \in G_{n_i}$) and upper triangular matrices. Let $N_{n_1, \ldots, n_s}$ be the unipotent radical of $P_{n_1, \ldots, n_s}$. 

We end with some conjectures for the embedding model, which are possibly used to interpret some results in this article from representation-theoretic viewpoint:

\begin{conjecture} \label{conj unique embedding}
Let $\mathfrak h \in \mathrm{Mult}_{\rho}$. The embedding in Proposition \ref{prop removal embedding} is unique. 
\end{conjecture}

We remark that Conjecture \ref{conj unique embedding} is not a mere consequence of the multiplicity one theorem for standard representations in \cite{Ch21+}.

\begin{conjecture} \label{conj multisegment}
Let $\mathfrak h \in \mathrm{Mult}_{\rho}$. Let $\mathfrak n \in \mathrm{Mult}_{\rho}$ be minimal to $\mathfrak h$. Let $l=l_{abs}(\mathfrak n)$. Then there exists a unique embedding:
\[   \widetilde{\lambda}(\mathfrak r(\mathfrak n, \mathfrak h)) \boxtimes \widetilde{\lambda}(\mathfrak n) \hookrightarrow \widetilde{\lambda}(\mathfrak h)_{N_l} .
\]
\end{conjecture}

Proposition \ref{prop removal embedding} is a special case of Conjecture \ref{conj multisegment}.

\begin{conjecture}
Let $\pi \in \mathrm{Irr}_{\rho}$. Let $\mathfrak n \in \mathrm{Mult}_{\rho}$ be minimal to $\pi$. Let $\mathfrak h=\mathfrak{hd}(\pi)$. Let $l_1=l_{abs}(\mathfrak n)$ and let $l_2=l_{abs}(\mathfrak{hd}(\pi))-l_{abs}(\mathfrak n)$. Suppose Conjectures \ref{conj minimal model} and \ref{conj multisegment} hold. We have the following diagram of maps:
\[      \xymatrix{      D_{\mathfrak h}(\pi)\boxtimes \widetilde{\lambda}(\mathfrak r(\mathfrak n, \pi)) \boxtimes \widetilde{\lambda}(\mathfrak n)      \ar@{^{(}->}[r]^{\iota_1} &   D_{\mathfrak h}(\pi)\boxtimes \widetilde{\lambda}(\mathfrak h)_{N_{l_2,l_1}} \ar@{^{(}->}[r]^{\iota_2} &        \pi_{N_{n-l_1-l_2, l_2, l_1}}  \\ 
   &    &      D_{\mathfrak n}(\pi)_{N_{n-l_1-l_2,l_2}} \boxtimes \widetilde{\lambda}(\mathfrak n)  \ar@{^{(}->}[u]^{\iota_3}  },
\]
where 
\begin{itemize}
\item $\iota_1$ is the map indeced from the one in Conjecture \ref{conj multisegment};
\item $\iota_2$ is the map induced from the unique embedding;$D_{\mathfrak h}(\pi)\boxtimes \widetilde{\lambda}(\mathfrak h)$ in Conjecture \ref{conj minimal model}
\item $\iota_3$ is the map induced from the unique embedding $D_{\mathfrak n}(\pi)\boxtimes \widetilde{\lambda}(\mathfrak n) \hookrightarrow \pi_{n-l_1, l_1}$ in Conjecture \ref{conj minimal model}.
\end{itemize}
Then $\iota_2\circ \iota_1$ factors through $\iota_3$. 
\end{conjecture}

\part{Appendices}

\section{Appendix A: Non-isomorphic derivatives} \label{s lemma on non-isom derivatives}

We prove a converse of Proposition \ref{prop injective for a pair} in this appendix. For a ladder representation or a generic representation $\sigma$ of $G_k$, let $I_{\sigma}(\pi)$ be the unique submodule of $\pi \times \sigma$ (see \cite{LM16}, also see \cite{Ch22+, Ch22+d}). See \cite{At24} for a generalization of ladder representations to some other classical groups.

\subsection{Non-isomorphic integrals}


For two segments $\Delta_1, \Delta_2$ in $\mathrm{Seg}_{\rho}$, 
\begin{enumerate}
\item Suppose $\Delta_1 <\Delta_2$. Define $\mathrm{St}(\left\{ \Delta_1, \Delta_2\right\})$ to be the unique irreducible quotient of $\mathrm{St}(\Delta_2)\times \mathrm{St}(\Delta_1)$.
\item Suppose $\Delta_1$ and $\Delta_2$ are not linked. Define $\mathrm{St}(\left\{ \Delta_1, \Delta_2\right\})$ to be the irreducible module $\mathrm{St}(\Delta_1) \times \mathrm{St}(\Delta_2)$. 
\end{enumerate}

\begin{proposition} \label{prop distinct sub ladder and generic}
Let $\pi \in \mathrm{Irr}_{\rho}$. Let $\Delta_1, \Delta_2$ be two linked segments. Let $\Delta_1'=\Delta_1\cup \Delta_2$ and let $\Delta_2'=\Delta_1\cap \Delta_2$ (possibly the empty set). Let $\sigma=\mathrm{St}(\left\{ \Delta_1, \Delta_2 \right\})$ and let $\sigma'=\mathrm{St}( \Delta_1'+\Delta_2' )$. Then $I_{\sigma}(\pi) \not\cong I_{\sigma'}(\pi)$. 
\end{proposition}

\begin{proof}
We shall use the invariant $\mathfrak{mx}(.,\Delta_1)$ to distinguish the two representations. Write $\Delta_1=[a_1,b_1]_{\rho}$ and $\Delta_2=[a_2,b_2]_{\rho}$. Switching labels if necessary, we may and shall assume $a_1<a_2$. We shall prove by an induction on the sum $l_{abs}(\Delta_1)+l_{abs}(\Delta_2)$. When the sum is $2$, the argument is similar to the cases below and we omit the details. \\

\noindent
{\bf Case 1:} $a_1\neq a_2-1$. Let $k=\varepsilon_{[a_1]_{\rho}}(\sigma)$. We have that $\varepsilon_{[a_1]_{\rho}}(I_{\sigma}(\pi)), \varepsilon_{[a_2]_{\rho}}(I_{\sigma'}(\pi))=\varepsilon_{[a_1]_{\rho}}(\pi)+1$. Let $\kappa=\mathrm{St}(\left\{{}^-\Delta_1,\Delta_2\right\})$ and let $\kappa'=\mathrm{St}(\left\{{}^-\Delta_1', \Delta_2'\right\}$. Furthermore,
\[   D_{[a_1]_{\rho}}^{k+1}(I_{\sigma}(\pi))=I_{\kappa}(D_{[a_1]_{\rho}}^k(\pi)), \quad D_{[a_1]_{\rho}}^{k+1}(I_{\sigma'}(\pi))=I_{\kappa'}(D_{[a_1]_{\rho}}^k(\pi)) .
\]
Then, by induction, we have that $I_{\kappa}(D_{[a_1]_{\rho}}^k(\pi)) \not\cong I_{\kappa'}(D_{[a_1]_{\rho}}^k(\pi))$ as desired. 

\noindent
{\bf Case 2:} $a_1 =a_2-1$. In this case, let $\mathfrak m=\mathfrak{mx}(\pi, \Delta_1)$. Then 
\[  \pi \hookrightarrow D_{\mathfrak m}(\pi) \times \mathrm{St}(\mathfrak m) \]
Now let $\widetilde{\sigma}=D_{\Delta_2'}(\sigma)$. We have that
\begin{align*}
  I_{\sigma}(\pi) & \hookrightarrow \pi \times \sigma \\
	                  & \hookrightarrow D_{\mathfrak m}(\pi)\times \mathrm{St}(\mathfrak m) \times \sigma  \\
									  & \cong D_{\mathfrak m}(\pi) \times \sigma \times \mathrm{St}(\mathfrak m)  \\
										& \hookrightarrow D_{\mathfrak m}(\pi) \times \widetilde{\sigma} \times \mathrm{St}(\Delta_2'+\mathfrak m) ,
\end{align*}
where the isomorphism in the third line follows from $\mathrm{St}(\Delta')\times \sigma \cong \sigma \times \mathrm{St}(\Delta')$ for any $\Delta_1$-saturated segment $\Delta'$ (see e.g. \cite[Lemme II 10.1]{MW86}).

Since $\mathfrak{mx}(D_{\mathfrak m}(\pi), \Delta_1)=\emptyset$ and $\mathfrak{mx}(\widetilde{\sigma},\Delta_1)=\emptyset$, we have that 
\[\mathfrak{mx}(I_{\sigma}(\pi), \Delta_1)=\mathfrak m+\Delta_2' \] 
The last equality follows from an application on the geometric lemma (see details from the proof of \cite[Proposition 11.1]{Ch22+}).

Similarly, let $\widetilde{\sigma}'=D_{\Delta_1}\circ D_{\Delta_2'}(\sigma')$. We also have that:
\begin{align*}
 I_{\sigma'}(\pi) & \hookrightarrow \pi \times \sigma' \\
                    & \hookrightarrow D_{\mathfrak m}(\pi) \times \mathrm{St}(\mathfrak m) \times \sigma' \\
										& \cong D_{\mathfrak m}(\pi) \times \sigma'\times \mathrm{St}(\mathfrak m)  \\
										& \hookrightarrow D_{\mathfrak m}(\pi) \times \widetilde{\sigma}' \times \mathrm{St}(\Delta_1) \times \mathrm{St}(\Delta_2') \times \mathrm{St}(\mathfrak m) \\
										& \hookrightarrow D_{\mathfrak m}(\pi) \times \widetilde{\sigma}' \times \mathrm{St}(\mathfrak m +\Delta_1 +\Delta_2') 
\end{align*}
Again,  $\mathfrak{mx}(D_{\mathfrak m}(\pi), \Delta_1)=\emptyset$ and $\mathfrak{mx}(\widetilde{\sigma}',\Delta_1)=\emptyset$. Thus, 
\[ \mathfrak{mx}(I_{\sigma'}(\pi), \Delta_1)=\mathfrak m+\Delta_1+\Delta_2' .
\]
Thus, comparing the invariant $\mathfrak{mx}(.,\Delta_1)$, we have $I_{\sigma}(\pi)\not\cong I_{\sigma'}(\pi)$.
\end{proof}

It is an interesting to investigate if an analogue of Proposition \ref{prop distinct sub ladder and generic} can be obtained if one replaces essentially square-integrable representations by other interesting representations such as Speh representations and ladder representations. The composition factors for parabolically induced from Speh representations and ladder representations are studied in \cite{Ta15} and \cite{Gu21}, and so it is possible to develop a parallel theory from those via above approach.

\subsection{Consequences}
We similarly define those notions for derivatives for ladder representations (also see e.g. \cite{Ch22+c}). If there exists $\omega \in \mathrm{Irr}(G_{n-k})$ such that $\omega \boxtimes \sigma \hookrightarrow \pi_{N_k}$ for $\sigma$ defined in Proposition \ref{prop distinct sub ladder and generic}, then denote such $\omega$ by $D_{\sigma}(\pi)$. Otherwise, set $D_{\sigma}(\pi)=0$. 

\begin{corollary} \label{cor non isomo derivatives}
We use the notations in Proposition \ref{prop distinct sub ladder and generic}. Then $D_{\sigma}(\pi)\not\cong D_{\sigma'}(\pi)$ if both terms are non-zero.
\end{corollary}

\begin{proof}
Let $\pi'=D_{\sigma}(\pi)$. Then $I_{\sigma}(\pi')\not\cong I_{\sigma'}(\pi')$ and so $\pi \not\cong I_{\sigma'}\circ D_{\sigma}(\pi)$. Applying $D_{\sigma'}$ on both sides, we obtain the corollary.
\end{proof}

\begin{corollary}
Let $\Delta_1, \Delta_2 \in \mathrm{Seg}_{\rho}$ such that $\Delta_1<\Delta_2$. Let $\pi \in \mathrm{Irr}_{\rho}$. Suppose $D_{\Delta_2}\circ D_{\Delta_1}(\pi)\neq 0$. If the non-zero map
\[  D_{\Delta_2}\circ D_{\Delta_1}(\pi) \boxtimes (\mathrm{St}(\Delta_1)\times \mathrm{St}(\Delta_2)) \rightarrow \pi_{N_{l_a(\Delta_1)+l_a(\Delta_2)}}
\]
is injective, then $\left\{ \Delta_1, \Delta_2 \right\}$ is minimal to $\pi$. 
\end{corollary}

\begin{proof}
If the map is injective, then $D_{\Delta_2}\circ D_{\Delta_1}(\pi)\boxtimes \mathrm{St}(\left\{\Delta_1+\Delta_2\right\})$ is a submodule of $\pi_{N_{l_a(\Delta_1)+l_a(\Delta_2)}}$. This implies that 
\[ D_{\Delta_2}\circ D_{\Delta_1}(\pi) \cong D_{\mathrm{St}(\Delta_1+\Delta_2)}(\pi).\] 
Then the corollary follows from Corollary \ref{cor non isomo derivatives}.
\end{proof}

\section{Appendix B: Applications}

\subsection{Minimality under $\Delta$-reduced condition}

\begin{corollary} \label{cor minimal reduced decomp}
Let $\pi \in \mathrm{Irr}_{\rho}$. Let $\mathfrak n \in \mathrm{Mult}_{\rho}$ be minimal to $\pi$. Let $\Delta$ be a segment. Suppose $b(\Delta')< b(\Delta)$ for any segment $\Delta' \in \mathfrak n$ with $b(\Delta')\neq b(\Delta)$. If $\eta_{\Delta}(D_{\mathfrak n}(\pi))=0$, then $\mathfrak{mx}(\pi, \Delta) \subseteq \mathfrak n$. 
\end{corollary}

\begin{proof}
Let $\Delta=[a,b]_{\rho}$. Let $\mathfrak p=\mathfrak{mx}(\pi, \Delta)$. Let $\mathfrak p'$ be all the segments $\Delta'$ in $\mathfrak n$ such that $b(\Delta')=b$ and $\Delta' \subseteq \Delta$. 

We first prove the following claim: \\
\noindent
{\it Claim:} $|\mathfrak p'|\geq |\mathfrak p|$. \\

\noindent
{\it Proof of Claim:} We also let $\mathfrak p''$ be all the segments $\Delta'$ in $\mathfrak n$ such that $b(\Delta')=b$ and $\Delta \subsetneq \Delta'$. Note that 
\[   \eta_{\Delta}(D_{\mathfrak p''} \circ D_{\mathfrak n-\mathfrak p''}(\pi))=\eta_{\Delta}(D_{\mathfrak n-\mathfrak p''}(\pi)) . \]
Hence, we also have $\eta_{\Delta}(D_{\mathfrak n-\mathfrak p''}(\pi))=0$. But then, by counting the multiplicity at $b(\Delta)$, we must have that $|\mathfrak p'|\geq |\mathfrak p|$. This proves the claim. \\

  If $\mathfrak p\neq \mathfrak p'$, then the Claim implies that there exists $a \leq c \leq b$ such that the number of segments $[c,b]_{\rho}$ in $\mathfrak p'$ is strictly greater than $\varepsilon_{[c,b]_{\rho}}(\pi)$. Let $c^*$ be such smallest integer. Now let $\widetilde{\mathfrak n}$ be the submultisegment of $\mathfrak n$ whose segments $\Delta'$ satisfy $a(\Delta')<a$. By the subsequent property, $\widetilde{\mathfrak n}+[c^*,b]_{\rho}$ and $\widetilde{\mathfrak n}+\mathfrak p'$ are still minimal to $\pi$. Now, by Proposition \ref{prop dagger property} and the minimality of $\widetilde{\mathfrak n}+[c^*,b]_{\rho}$, $\varepsilon_{[c^*,b]_{\rho}}(\mathfrak r(\widetilde{\mathfrak n}, \pi))=\varepsilon_{[c^*,b]_{\rho}}(\pi)$. But, by the admissibility of $\widetilde{\mathfrak n}+\mathfrak p'$, 
	\[ \varepsilon_{[c^*,b]_{\rho}}(\mathfrak r(\widetilde{\mathfrak n}, \pi))=\mbox{number of segments $[c^*,b]_{\rho}$ in $\mathfrak p'$} >\varepsilon_{[c^*,b]_{\rho}}(\pi). \]
This gives a contradiction. Thus we must have $\mathfrak p'=\mathfrak p$.
\end{proof}

\subsection{Generalized reduced decomposition}
The notion of reduced decomposition is introduced in \cite[Section 7]{AL23} for a segment $\Delta$.

We now describe a generalization of reduced decomposition for multisegments. Let $\pi \in \mathrm{Irr}_{\rho}(G_n)$. Let $\mathfrak n$ be minimal to $\pi$. Let $b$ be the largest $b(\Delta)$ among all segments $\Delta$ in $\mathfrak n$. We then choose the longest segment $\Delta_1 \in \mathfrak n$ such that $b(\Delta_1)=b$. Let $\mathfrak p_1=\mathfrak{mx}(D_{\mathfrak n}(\pi), \Delta)$. We now set $\mathfrak m_1'=\mathfrak n+\mathfrak p_1$. In general, $\mathfrak n_1$ is not minimal to $\pi$ and so one find the minimal element, denoted by $\mathfrak m_1$, in $\mathcal S(\pi, D_{\mathfrak m_1}(\pi))$. Then, by Corollary \ref{cor minimal reduced decomp}, $\mathfrak m_1=\mathfrak q_1+\mathfrak n_2$ for some multisegment $\mathfrak n_2$ and $\mathfrak q_1=\mathfrak{mx}(\pi, \Delta)$. Thus, from commutativity, we now have that $D_{\mathfrak m_1}(\pi) \cong D_{\mathfrak n_2}\circ D_{\mathfrak q_1}(\pi)$. Thus, one may consider that $\pi \hookrightarrow D_{\mathfrak q_1}(\pi) \times \mathrm{St}(\mathfrak q_1)$ is the step for the reduction. One can repeat the same process for $\mathfrak n_2$ and then repeat to obtain a sequence of triples $(\mathfrak p_1, \mathfrak q_1, \mathfrak n_1), \ldots, (\mathfrak p_r, \mathfrak q_r, \mathfrak n_r)$ until the process terminates. Then we obtain a kind of reduced decompositon for $\pi$ with respect to $\mathfrak n$ as follows:
\[  \pi \hookrightarrow   (D_{\mathfrak q_r}\circ \dots \circ D_{\mathfrak q_1})(\pi) \times \mathrm{St}(\mathfrak q_r)\times \ldots \times \mathrm{St}(\mathfrak q_1) .
\]
Let $l=l_{abs}(\mathfrak n)$. One may expect there is a natural map:
\[    D_{\mathfrak n}(\pi)\boxtimes \mathrm{St}(\mathfrak n) \rightarrow (D_{\mathfrak q_r}\circ \dots \circ D_{\mathfrak q_1})(\pi) \dot{\times}^1 (\mathrm{St}(\mathfrak q_r)\times \ldots \times \mathrm{St}(\mathfrak q_1))_{N_l},
\]
where $\dot{\times}^1$ is a parabolic induction from a $G_{n_1}\times G_{n_2}\times G_l$-representation to a $G_{n_1+n_2}\times G_l$ representation. Here $n_1=n-l_{abs}(\mathfrak q_1)-\ldots-l_{abs}(\mathfrak q_r)$ and $n_2=l_{abs}(\mathfrak p_1)+\ldots +l_{abs}(\mathfrak p_r)$.

\subsection{An inductive construction of simple quotients of Bernstein-Zelevinsky derivatives}


Let $\pi \in \mathrm{Irr}_{\rho}$. Let $\Delta \in \mathrm{Seg}_{\rho}$ and let $\mathfrak n \in \mathrm{Mult}_{\rho}$. Let $\mathfrak p=\mathfrak{mx}(\pi, \Delta)$. Note that $D_{\mathfrak n}\circ D_{\mathfrak p}(\pi)$ is a simple quotient of the $l_{abs}(\mathfrak n+\mathfrak p)$-th Bernstein-Zelevinsky derivative of $\mathrm{St}(\mathfrak p)\times D_{\mathfrak p}(\pi)$ if $\eta_{\Delta}(D_{\mathfrak n}\circ D_{\mathfrak p}(\pi))=0$. Then one may ask if $D_{\mathfrak n}\circ D_{\mathfrak p}(\pi)$ is also a simple quotient of the $l_{abs}(\mathfrak n+\mathfrak p)$-th Bernstein-Zelevinsky derivative of $\pi$. In a special case, we have the following criteria using commutativity:

\begin{proposition} \label{prop inductive construction}
Let $\pi \in \mathrm{Irr}_{\rho}$. Let $\Delta \in \mathrm{Seg}_{\rho}$ and let $\mathfrak n \in \mathrm{Mult}_{\rho}$. Let $\mathfrak p=\mathfrak{mx}(\pi, \Delta)$.
\begin{enumerate}
\item[(1)] Suppose $\mathfrak n$ is minimal to $D_{\mathfrak p}(\pi)$ and $\eta_{\Delta}(D_{\mathfrak n}\circ D_{\mathfrak p}(\pi))=0$. Suppose further that for any segment $\Delta' \in\mathfrak n$, $b(\Delta')<b(\Delta)$. If $D_{\mathfrak n}\circ D_{\mathfrak p}(\pi) \cong D_{\mathfrak m}(\pi)$ for some $\mathfrak m \in \mathrm{Mult}_{\rho}$, then  $\mathfrak n+\mathfrak p$ is minimal to $\pi$. 
\item[(2)] Suppose $\mathfrak m$ is minimal to $\pi$ and $\eta_{\Delta}(D_{\mathfrak m}(\pi))=0$. Suppose further that for any segment $\Delta' \in\mathfrak m$, $b(\Delta') \leq b(\Delta)$. Then $D_{\mathfrak m}(\pi)\cong D_{\mathfrak n}\circ D_{\mathfrak p}(\pi)$ and $\mathfrak m=\mathfrak n+\mathfrak p$ for some multisegment $\mathfrak n$ minimal to $D_{\mathfrak p}(\pi)$.
\end{enumerate}
\end{proposition}

\begin{proof}
We first consider (1). Suppose $D_{\mathfrak n}\circ D_{\mathfrak p}(\pi)\cong D_{\mathfrak m}(\pi)$ for some $\mathfrak m \in \mathrm{Mult}_{\rho}$. Without loss of generality, we may assume that $\mathfrak m$ is minimal to $\pi$. Then, by Corollary \ref{cor minimal reduced decomp}, $\mathfrak m=\mathfrak p+\mathfrak n'$. By the subsequent property and commutativty property (Theorem \ref{thm minimal and commut 2}), $\mathfrak n'$ is also minimal to $D_{\mathfrak p}(\pi)$ and $D_{\mathfrak n'}\circ D_{\mathfrak p}(\pi)\cong D_{\mathfrak m}(\pi)$. By Theorem \ref{thm unique minimal}, we then have that $\mathfrak n=\mathfrak n'$ and this implies (1).

We now consider (2). By Corollary \ref{cor minimal reduced decomp}, $\mathfrak p \subset \mathfrak m$ and so $\mathfrak m=\mathfrak p+\mathfrak n$ for some multisegment $\mathfrak n$. By the subsequent property and commutativity property (Theorem \ref{thm minimal and commut 2}), we also have that $\mathfrak n$ is minimal to $D_{\mathfrak p}(\pi)$.
\end{proof}

\end{document}